\documentclass[11pt]{amsart}
\usepackage{hyperref}
\usepackage{amsfonts}
\usepackage{amsmath}
\usepackage{amssymb}
\usepackage{amscd}
\usepackage{graphicx}
\usepackage{enumitem}
\usepackage{latexsym}
\usepackage{color}
\usepackage{epsfig}
\usepackage{amsopn}
\usepackage{xypic}
\usepackage{mathrsfs}
\usepackage{array}

\usepackage{booktabs}
\usepackage{longtable}
\theoremstyle{plain}
\newtheorem{lemma}{Lemma}[section]
\newtheorem*{theorem*}{Theorem}
\newtheorem*{lemma*}{Lemma}
\newtheorem*{proposition*}{Proposition}
\newtheorem*{conjecture*}{Conjecture}
\newtheorem*{corollary*}{Corollary}
\newtheorem*{problem*}{Problem}
\newtheorem{theorem}[lemma]{Theorem}

\newtheorem{corollary}[lemma]{Corollary}
\newtheorem{proposition}[lemma]{Proposition}

\theoremstyle{definition}
\newtheorem{definition}[lemma]{Definition}

\newtheorem{remark}[lemma]{Remark}

\newcommand{\Z}{\mathbb{Z}}

\newcommand{\N}{\mathbb{N}}

\newcommand{\C}{\mathbb{C}}
\newcommand{\Q}{\mathbb{Q}}
\newcommand{\R}{\mathbb{R}}

\newcommand{\OO}{\mathcal{O}}
\newcommand{\te}{\otimes}

\newcommand{\cF}{\mathcal F}

\newcommand{\cA}{\mathcal A}

\newcommand{\cO}{\mathcal{O}}
\newcommand{\cP}{\mathcal P}
\newcommand{\cT}{\mathcal{T}}
\newcommand{\cE}{\mathcal{E}}
\renewcommand{\cD}{\mathcal{D}}
\newcommand{\rH}{\mbox{H}}

\newcommand{\ZZ}{\mathbb{Z}}
\newcommand{\leqpar}{\underset{{\scriptscriptstyle (}-{\scriptscriptstyle )}}{<}}
\newcommand{\geqpar}{\underset{{\scriptscriptstyle (}-{\scriptscriptstyle )}}{>}}

\renewcommand{\P}{\mathbb{P}}
\renewcommand{\cL}{\mathcal{L}}
\newcommand{\PP}{\mathbb{P}}

\DeclareMathOperator{\ch}{ch}

\DeclareMathOperator{\Aut}{Aut}

\DeclareMathOperator{\Hom}{Hom}

\DeclareMathOperator{\Pic}{Pic}

\DeclareMathOperator{\NS}{NS}

\DeclareMathOperator{\rk}{rk}

\DeclareMathOperator{\Ext}{Ext}
\DeclareMathOperator{\ext}{ext}

\DeclareMathOperator{\Coh}{Coh}

\DeclareMathOperator{\Stab}{Stab}
\DeclareMathOperator{\Amp}{Amp}

\DeclareMathOperator{\Nef}{Nef}

\DeclareMathOperator{\num}{num}

\begin{document}

\date{\today}
\author[I. Coskun]{Izzet Coskun}
\address{Department of Mathematics, Statistics and CS \\University of Illinois at Chicago, Chicago, IL 60607}
\email{coskun@math.uic.edu}
\author[J. Huizenga]{Jack Huizenga}
\address{Department of Mathematics, The Pennsylvania State University, University Park, PA 16802}
\email{huizenga@psu.edu}
\thanks{During the preparation of this article the first author was partially supported by the NSF CAREER grant DMS-0950951535 and NSF grant DMS-1500031, and the second author was partially supported by a National Science Foundation Mathematical Sciences Postdoctoral Research Fellowship}

\title[The nef cone of the moduli space of sheaves]{The nef cone of the moduli space of sheaves and strong Bogomolov inequalities}

\begin{abstract}
Let $(X,H)$ be a polarized, smooth, complex projective surface, and let ${\bf v}$ be a Chern character on $X$ with positive rank and sufficiently large discriminant. In this paper, we compute the Gieseker wall for ${\bf v}$ in a slice of the stability manifold of $X$. We construct explicit curves parameterizing non-isomorphic  Gieseker stable sheaves of character ${\bf v}$ that become $S$-equivalent along the wall. As a corollary, we conclude that if there are no strictly semistable sheaves of character ${\bf v}$, the Bayer-Macr\`{i} divisor associated to the wall is a boundary nef divisor on the moduli space of sheaves $M_H({\bf v})$.  We recover previous results for $\P^2$ and $K3$ surfaces, and illustrate applications to higher Picard rank surfaces with an example on $\P^1\times \P^1$.
\end{abstract}
\maketitle

\setcounter{tocdepth}{1}
\tableofcontents

\section{Introduction}
Let $(X,H)$ be a polarized, smooth, complex projective surface. Let $D$ be a $\Q$-divisor on $X$ and let ${\bf v} \in K_{\num}(X)$ be the class of a stable sheaf with positive rank and sufficiently large discriminant. The divisors $H$ and $D$ determine a half-plane in the Bridgeland stability manifold $\Stab(X)$ called the $(H,D)$-slice (see \S \ref{ssec-Bridgeland} for the precise definition). Let $M_{H,D}({\bf v})$ denote the moduli space parameterizing $(H,D)$-twisted Gieseker semistable sheaves with Chern character ${\bf v}$. The large volume limit of the Bridgeland moduli spaces in the $(H,D)$-slice is  $M_{H,D}({\bf v})$. Let the  Gieseker wall $W$ be the largest Bridgeland wall in the slice along which an $(H,D)$-semistable sheaf with Chern character ${ \bf v}$ is destabilized.

In this paper, we compute the Gieseker wall $W$ for ${ \bf v}$ in the $(H,D)$-slice. We also construct an explicit curve parameterizing non-isomorphic $(H,D)$-twisted Gieseker stable sheaves that become $S$-equivalent for the Bridgeland stability conditions along $W$.  As a corollary, we conclude that if $M_{H,D}({\bf v})$ has no strictly semistable sheaves, then the Bayer-Macr\`{i} divisor constructed in \cite[Lemma 3.3]{BayerMacri2} is a nef divisor which lies in the boundary of $\Nef(M_{H,D}({\bf v}))$. When $\Pic(X) \cong \Z$, we show that the problem of computing Gieseker walls in $(H,D)$-slices for large discriminants is equivalent to the problem of classifying stable Chern characters.   Our computations recover previous results for $\P^2$ \cite{CoskunHuizenga2} and K3 surfaces \cite{BayerMacri2,BayerMacri3}.  We also explore new applications in the setting of $\P^1\times \P^1$, surfaces in $\P^3$ and double covers of $\P^2$.

Let $Y$ be a projective variety. The ample cone $\Amp(Y) \subset N^1(Y)$ is the open convex cone in the N\'{e}ron-Severi space spanned by the classes of ample divisors. It encodes  embeddings of $Y$ in projective space and is among the most important invariants of $Y$. The closure of $\Amp(Y)$ is the nef cone $\Nef(Y) \subset N^1(Y)$ spanned by the classes of divisors that have nonnegative intersection with every integral curve on $Y$. By definition, $\Nef(Y)$ is dual to the Mori cone of curves (see \cite{Lazarsfeld}). Computing $\Nef(Y)$ requires finding nef divisors to generate a subcone of $\Nef(Y)$ and dually finding integral curves on $Y$ to bound $\Nef(Y)$ from above. In this paper, we carry out this strategy for  $\Nef(M_{H, D}({\bf v}))$ when $M_{H,D}({\bf v})$ contains only stable sheaves.

We now explain our strategy in greater detail. Given a Chern character ${\bf v}$,  we define an {\em extremal Chern character} ${\bf w}$.  Intuitively, ${\bf w}$ is chosen to make the numerical wall $W({\bf w},{\bf v})$ that it determines as large as possible, subject to natural restrictions which will ensure that the wall is an actual wall where a semistable sheaf is destabilized.   See Definition \ref{def-extremal} for the precise conditions defining ${\bf w}$.  If sufficiently strong Bogomolov-type results are known about the Chern characters of stable sheaves on $X$, then ${\bf w}$ can be computed explicitly.

\begin{theorem}
Assume that the discriminant $\Delta_{H,D} ({\bf v}) \gg 0$. Then the Gieseker wall for ${\bf v}$ in the $(H,D)$-slice is given by the wall $W({\bf w}, {\bf v})$. 
\end{theorem}

Throughout the paper we will consider $r({\bf v})$, $c_1({\bf v})$, $X$, $H$, and $D$ as fixed, and $\Delta_{H,D}({\bf v})$ as variable.  Thus, we write $\Delta_{H,D}({\bf v})\gg 0$ to mean that $\Delta_{H,D}({\bf v})>C$ for some constant $C$ depending on $r({\bf v})$, $c_1({\bf v})$, $X$, $H$, and $D$.

A numerical check, using inequalities  discovered in \cite{ABCH}, \cite{CoskunHuizenga2} and further elaborated in \cite{Betal}, confirms that the Gieseker wall is contained in $W({\bf w}, {\bf v})$. Conversely, we need to construct a family of non-isomorphic Gieseker stable sheaves that become $S$-equivalent for the Bridgeland stability conditions on $W({\bf w}, {\bf v})$. We construct the necessary sheaves inductively. The key is to show that the Gieseker walls corresponding to ${\bf w}$ and the quotient Chern character ${\bf u} = {\bf v} - {\bf w}$ are both nested inside $W({\bf w}, {\bf v})$. One can then inductively construct the required curves (see \S \ref{sec-curves}). 

Given two Chern characters ${\bf z}$ and ${\bf v}$, define  the incidence variety 
$$Z({\bf z}) := \{ (E,F) \in  M_{H, D}({\bf z}) \times M_{H, D}({\bf v}) | \Hom(E,F) \not=0 \}.$$ In general, describing the geometric  properties of $Z({\bf z})$, even determining when they are nonempty, can be very challenging. Our results in \S \ref{sec-curves} imply that $Z({\bf w})$ is nonempty for the extremal Chern character ${\bf w}$ provided the discriminant of ${\bf v}$ is sufficiently large. 

\subsection*{Applications to birational geometry} For the rest of the introduction, we additionally assume that ${\bf v}$ is a Chern character such that $M_{H, D}({\bf v})$ contains only stable sheaves.  In this case, Bayer and Macr\`i \cite{BayerMacri2} associate a nef divisor to the Gieseker wall.

\begin{corollary}\label{cor-main}
If $M_{H, D}({\bf v})$ contains only stable sheaves and $\Delta_{H,D}({\bf v})\gg 0$, then the Bayer-Macr\`i divisor associated to $W({\bf w}, {\bf v})$ is a nef divisor lying in the boundary of  $\Nef(M_{H, D}({\bf v}))$.
\end{corollary}

Computing the Gieseker wall is closely related to Bogomolov-type inequalities for $(H,D)$-semistable sheaves on $X$.  The correspondence is tightest when $\Pic(X) \cong \Z H$ with $H$ effective, so we focus on this case.  Then the determinantal line bundles on $M_H({\bf v})$ span a $2$-dimensional subspace of $N^1(M_H({\bf v}))$.  The intersection of $\Nef(M_H({\bf v}))$ with this determinantal subspace is a cone spanned by two classes.  One of these classes $\cL_1$ corresponds to the Donaldson-Uhlenbeck-Yau compactification by slope-semistable sheaves. When $\Delta_{H,D}({\bf v})\gg 0$, then there are singular sheaves in $M_{H,D}({\bf v})$ and the map to the twisted Donaldson-Uhlenbeck-Yau space is not an isomorphism (\cite{HuybrechtsLehn},\cite{Greb}). Let $\cL_2$ span the other extremal ray.  Given a rank $r$ and slope $\mu = c_1/r$, let $\delta(r,\mu)$ denote the minimal discriminant of a semistable sheaf of slope $\mu$ and rank at most $r$.  Then the inequality $$\Delta \geq \delta(r,\mu)$$ holds for any semistable sheaf with invariants $(r,\mu,\Delta)$.  This refines the ordinary Bogomolov inequality $\Delta\geq 0$.

\begin{corollary}\label{cor-BogIntro}
Assume that $\Pic(X) \cong \Z H$ with $H$ effective.  The computation of $\cL_2$ for all characters ${\bf v}$ with $\Delta({\bf v})\gg 0$ is equivalent to the computation of the function $\delta(r,\mu)$ for all $r>0$, $\mu\in \Q$ with $r\mu \in \Z$.
\end{corollary}

In the higher Picard rank case, global information about the nef cone can frequently be obtained by varying the twisting divisor $D$.  If ${\bf v}$ is a Chern character such that $H$-Gieseker semistability and $\mu_H$-slope stability coincide, then both of these notions are also equivalent to $(H,D)$-Gieseker stability for any choice of twisting divisor $D$.  Varying the twisting divisor, the various Bridgeland $(H,D)$-slices correspond to rays in $N^1(M_H({\bf v}))$ by the Bayer-Macr\`i construction.  The corresponding extremal nef divisors vary as well.  This method was used in \cite{Betal} to compute the nef cone of the Hilbert scheme of points on a del Pezzo surface of degree $1$.  In this paper we will illustrate this method by studying a moduli space of sheaves on $\P^1\times \P^1$ in detail.

Bridgeland stability has been successfully used by many authors to study the birational geometry of moduli spaces of sheaves on surfaces. We refer the reader to  \cite{ABCH}, \cite{BertramMartinezWang}, \cite{CoskunHuizenga2},  \cite{CoskunHuizengaWoolf}, \cite{Ohkawa}, \cite{LiZhao} for $\PP^2$, \cite{BertramCoskun} for Hirzebruch and del Pezzo surfaces, \cite{BayerMacri2}, \cite{BayerMacri3}, \cite{MYY1}, \cite{MYY2} for K3 surfaces, \cite{Nuer} for Enriques surfaces, and \cite{MM}, \cite{YanagidaYoshioka}, \cite{Yoshioka2} for abelian surfaces. The ample cones of Hilbert schemes of points on surfaces were studied in \cite{Betal}. This paper generalizes and unifies the techniques in these papers for computing the ample cones to moduli spaces of sheaves on arbitrary surfaces. In a parallel development, the papers \cite{BertramMartinez} and \cite{Yoshioka3} study Thaddeus flips resulting from change of polarization in terms of Bridgeland stability. 

\subsection*{Organization of the paper} In \S \ref{sec-preliminaries}, we introduce the necessary background on $M_{H,D}({\bf v})$ and Bridgeland stability. In \S \ref{sec-destabilizing}, we introduce the extremal Chern character ${\bf w}$  associated to ${\bf v}$ and state our main result. In \S \ref{sec-Gieseker}, we show that the Gieseker wall is no larger than $W({\bf w},{\bf v})$ if the discriminant of ${\bf v}$ is sufficiently large. In \S \ref{sec-nesting}, we show that the Gieseker walls corresponding to ${\bf w}$ and ${\bf u} = {\bf v}-{\bf w}$ are nested in $W({\bf w}, {\bf v})$. Using the nesting result, in \S \ref{sec-curves}, we construct curves of Gieseker stable sheaves which become $S$-equivalent along $W({\bf w},{\bf v})$.  In \S \ref{sec-examples}, we study the nef cone of moduli spaces of rank $2$ sheaves several families of surfaces.

\subsection*{Acknowledgements} We would like to thank Arend Bayer, Aaron Bertram, Emanuele Macr\`{i},  Benjamin Schmidt and Matthew Woolf for many enlightening discussions on Bridgeland stability. Part of this work was carried out during the Algebraic Geometry Summer Research Institute in Utah. We thank the American Mathematical Society, the University of Utah and the organizers for providing us with ideal working conditions.

\section{Preliminaries}\label{sec-preliminaries}

In this section,  we review basic facts concerning moduli spaces of Gieseker semistable sheaves and Bridgeland stability conditions. 

\subsection{Basic definitions}
We refer the reader to \cite{HuybrechtsLehn} and \cite{MatsukiWentworth} for an in depth treatment of (twisted) Gieseker semistability.  

Let $X$ be a smooth projective surface over $\C$. A sheaf on $X$ will always mean a coherent sheaf of pure dimension. Fix an ample divisor $H \in \Pic(X)$. For any  $\Q$-divisor $D$ on $X$,  define the twisted Chern character $\ch^D = e^{-D} \ch$, with expansion $$\ch_0^D = \ch_0, \quad \ch_1^D = \ch_1 - D \ch_0, \quad \ch_2^D = \ch_2 - D \ch_1 + \frac{D^2}{2} \ch_0.$$ We will find it convenient to work with coordinates provided by the slope and the discriminant. For a sheaf $E$ with $\rk(E)>0$, define the $(H,D)$-slope $\mu_{H,D}$ and $(H,D)$-discriminant $\Delta_{H,D}$ by $$\mu_{H,D} = \frac{H \cdot \ch_1^D}{H^2 \ch_0^D}, \quad \Delta_{H,D} = \frac{1}{2} \mu_{H,D}^2 - \frac{\ch_2^D}{H^2 \ch_0^D}.$$   A sheaf $E$ of positive rank is $\mu_{H,D}$-(semi)stable  if for every nonzero proper subsheaf $F \subset E$, $$\mu_{H,D}(F) \leqpar \mu_{H,D}(E).$$ Note that $\mu_{H,D}$ only differs from $\mu_{H} := \mu_{H,0}$ by a constant, so $\mu_{H,D}$-(semi)stability and $\mu_{H}$-(semi)stability coincide.

A slightly modified version of these invariants will often be more useful.  We define $\overline{\ch}^D = \ch^{D+\frac{1}{2}K_X}$, and define modified slopes $\overline{\mu}_{H,D}$ and discriminants $\overline{\Delta}_{H,D}$ using this modified Chern character; thus $$\overline{\mu}_{H,D} = \mu_{H,D+\frac{1}{2}K_X} \qquad\mbox{and}\qquad \overline{\Delta}_{H,D}=\Delta_{H,D+\frac{1}{2}K_X}.$$ The additional twist by $\frac{1}{2}K_X$ (especially in the discriminant) greatly simplifies computations with twisted Gieseker semistability, which we now discuss.  The {\em reduced twisted Hilbert polynomial} of a positive rank sheaf $E$ is defined by $$p_{H,D}^E(m) = \frac{\chi(E \otimes \OO_X(mH-D))}{\rk(E)},$$ where the Euler characteristic is computed formally. A sheaf $E$ is {\em $(H,D)$-twisted Gieseker (semi)stable} if for every nonzero proper subsheaf $F \subset E$, $$p_{H,D}^F(m) \leqpar p_{H,D}^E(m)$$ for all $m\gg 0$.  When $D=0$, we recover usual $H$-Gieseker semistability.  When referring to $H$-Gieseker semistability, we will omit $D$ from our notation. A simple Riemann-Roch computation shows that a sheaf $E$ is $(H,D)$-twisted Gieseker (semi)stable if and only if
\begin{enumerate}
\item $E$ is $\mu_{H,D}$-semistable, and 
\item if $F\subsetneq E$ with $\mu_{H,D}(F) = \mu_{H,D}(E)$, then $$\overline{\Delta}_{H,D}(F)\geqpar \overline{\Delta}_{H,D}(E).$$
\end{enumerate}
Note that this equivalence is not typically correct when using the ordinary discriminant $\Delta_{H,D}$ instead of $\overline\Delta_{H,D}$ unless $K_X$ is parallel to $H$.  

\begin{remark}
Note that when $r({\bf v})$, $c_1({\bf v})$, $X$, $H$, and $D$ are fixed, we have $\Delta_{H,D}({\bf v}) \gg 0$ if and only if $\overline{\Delta}_{H,D}({\bf v}) \gg 0$.  Thus the results in the introduction could also be stated in terms of the discriminant $\overline \Delta_{H,D}({\bf v})$ instead of $\Delta_{H,D}({\bf v})$.
\end{remark}

\subsection{Properties of the moduli space}
Recall that two semistable sheaves are {\em $S$-equivalent} with respect to a notion of stability if they have the same Jordan-H\"{o}lder factors for that stability condition.  Fix the Chern character ${\bf v} $ of an $(H,D)$-twisted Gieseker semistable sheaf. Matsuki and Wentworth \cite{MatsukiWentworth} prove that there are projective moduli spaces $M_{H,D}({\bf v})$ parameterizing $S$-equivalence classes of $(H,D)$-twisted Gieseker semistable sheaves on $X$ with invariants ${\bf v}$. 

Recall the following fundamental theorems of O'Grady for the ordinary Gieseker moduli space.

\begin{theorem}[{\cite[Theorems B, D]{OGrady}, \cite[Theorems 5.2.5, 9.3.3, 9.4.3]{HuybrechtsLehn}}]\label{thm-OGrady}
Let $(X,H)$ be a smooth polarized surface, and let ${\bf v}\in K(X)$ with $r({\bf v})>0$. If $\Delta_H({\bf v})\gg 0$ (depending on $r({\bf v})$, $X$, and $H$), then the moduli space $M_H({\bf v})$ is normal, generically smooth, irreducible, and nonempty of the expected dimension.
Furthermore, the slope stable sheaves are dense in $M_H({\bf v})$.
\end{theorem}

As we remarked earlier, the two slopes $\mu_H$ and  $\mu_{H,D}$ differ only by a constant. Consequently, $\mu_H$- and $\mu_{H,D}$-stability coincide. Since $\mu_{H,D}$-stability is an open condition, Theorem \ref{thm-OGrady} provides a nonempty Zariski open set in $M_{H,D}({\bf v})$ parameterizing $\mu_{H,D}$-stable sheaves.

\subsection{Bridgeland stability}\label{ssec-Bridgeland} In this subsection, we  review key facts concerning Bridgeland stability on surfaces. We refer the reader to \cite{ArcaraBertram}, \cite{ABCH}, \cite{BayerMacri2}, \cite{Betal}, \cite{bridgeland:stable}, \cite{Bridgeland}, \cite{CoskunHuizenga} and \cite{CoskunHuizenga2} for more details.

Let $\cD^b(X)$ denote the bounded derived category of coherent sheaves on $X$. A {\em Bridgeland stability condition} on $\cD^b(X)$ is a pair $\sigma=(Z, \cA)$, where $\cA$ is the heart of a bounded $t$-structure on $\cD^b(X)$ and $Z:K_0(X) \rightarrow \C$ is a group homomorphism mapping $\cA$ to the extended upper half plane and satisfying the Harder-Narasimhan and Support Properties \cite{bridgeland:stable}, \cite{BayerMacri2}. The set $\Stab(X)$ of Bridgeland stability conditions on $X$ is a complex manifold \cite{bridgeland:stable}.

Bridgeland \cite{Bridgeland} and Arcara, Bertram \cite{ArcaraBertram} constructed Bridgeland stability conditions on surfaces. Given an ample divisor $H$, an arbitrary $\R$-divisor $D$ and  $\beta \in \R$, define two subcategories of the category of coherent sheaves $\Coh(X)$ by
\begin{align*}\cT_{\beta} &= \{ E \in \Coh(X) : \overline{\mu}_{H,D}(G) > \beta \ \mbox{for every quotient} \ G \ \mbox{of} \ E\} \\
\cF_{\beta} &=\{ E \in \Coh(X) :  \overline{\mu}_{H,D}(F) \leq \beta  \ \mbox{for every subsheaf} \ F \ \mbox{of} \ E \}.\end{align*}
The pair $(\cT_{\beta}, \cF_{\beta})$ forms a torsion pair in $\Coh(X)$. Tilting  $\Coh(X)$ with respect to this torsion pair yields the heart $\cA_{\beta}$ of  a new bounded $t$-structure on $\cD^b(X)$ defined by
$$\cA_{\beta} = \{ E^{\bullet} \in \cD^b(X) : H^{-1}(E^{\bullet}) \in \cF_{\beta}, H^0(E^{\bullet}) \in \cT_{\beta}, H^i(E^{\bullet})=0, i\not= -1,0\}.$$

Let $\alpha$ be a positive real number. Define the central charge
$$Z_{\beta, \alpha} = -\overline{\ch}_2^{D+\beta H} + \frac{\alpha^2 H^2}{2} \overline{\ch}_0^{D +\beta H} + i H \overline{\ch}_1^{D + \beta H}.$$
Then the pair $\sigma_{\beta, \alpha} = (Z_{\beta, \alpha}, \cA_{\beta})$ is a Bridgeland stability condition for $\alpha, \beta \in \R, \alpha >0$ \cite{ArcaraBertram}. These stability conditions span a half-plane in $\Stab(X)$ which we call {\em the $(H,D)$-slice}.  The $\sigma_{\beta, \alpha}$-slope of an object with invariants $r > 0$, $\overline{\mu}_{H,D}$ and $\overline{\Delta}_{H,D}$ is given by $$\nu_{\sigma_{\beta, \alpha}} = - \frac{\Re Z_{\beta, \alpha}}{\Im Z_{\beta, \alpha}} = \frac{(\overline\mu_{H,D} - \beta)^2 - \alpha^2 - 2\overline\Delta_{H,D}}{\overline\mu_{H,D} - \beta}.$$

\subsubsection{Bridgeland Walls} Fix an invariant ${\bf v} \in K_{\num}(X)$. Assume that ${\bf w} \in K_{\num}(X)$ is an invariant that does not have the same $\sigma_{\beta, \alpha}$-slope as ${\bf v}$ everywhere in the $(H,D)$-slice. Then the {\em numerical wall}  $W({\bf w}, {\bf v})$ is the set of points $(\beta, \alpha)$ such that ${\bf v}$ and ${\bf w}$ have the same $\sigma_{\beta, \alpha}$-slope.  A numerical wall is an {\em actual} wall  if there exists a point $(\beta, \alpha) \in W({\bf w, v})$ and an exact sequence 
$$0 \to F \to E \to G \to 0$$ in $\cA_{\beta}$ with $\ch(F) = {\bf w}$, $\ch(E)= {\bf v}$ such that $E,F,G$ are $\sigma_{\beta, \alpha}$-semistable. We will refer to the sequence as the {\em destabilizing sequence}. We will frequently use the following facts about the Bridgeland walls \cite{Betal}, \cite{CoskunHuizenga}, \cite{Maciocia}:
\begin{enumerate}
\item The numerical walls  $W({\bf w},{\bf v})$  in the $(H,D)$-slice are disjoint. Let ${\bf v}$ and ${\bf w}$ have positive rank. If $\overline\mu_{H,D}({\bf v})= \overline\mu_{H,D}({\bf w})$, then $W({\bf w},{\bf v})$ is the vertical wall $\beta= \overline\mu_{H,D}({\bf v})$. If $\overline\mu_{H,D}({\bf v}) \neq \overline\mu_{H,D}({\bf w})$, then $W({\bf w},{\bf v})$ is the semicircular wall with center $(s,0)$ and radius $\rho$, where
$$s = \frac{1}{2} (\overline\mu_{H,D}({\bf v}) + \overline\mu_{H,D}({\bf w})) - \frac{\overline\Delta_{H,D}({\bf v}) - \overline\Delta_{H,D}({\bf w})}{\overline\mu_{H,D} ({\bf v})- \overline\mu_{H,D}({\bf w})},$$ 
$$\rho^2 = (s - \overline\mu_{H,D}({\bf v}))^2 - 2 \overline\Delta_{H,D}({\bf v}).$$
If $\rho^2$ is negative, then the wall is empty.
\item Let $W_1, W_2$ be two numerical walls to the  left  of  $ \beta = \overline\mu_{H,D}({\bf v})$ with centers $(s_1, 0)$, $(s_2,0)$. Then $W_1$ is nested inside $W_2$ if and only if $ s_1 > s_2$.
\item Let $W({\bf w, v})$ be an actual wall. If $0\to F \to E \to G \to 0$ is a destabilizing sequence at a point $(\beta, \alpha) \in W({\bf w, v})$, then it is a destabilizing sequence for every point of $W({\bf w, v})$.
\end{enumerate}

Define the {\em Gieseker wall} for $\bf{v}$ to be the largest actual semicircular wall to the left of $\beta = \overline\mu_{H,D}({\bf v})$ where a Gieseker semistable sheaf is destabilized. In this paper, we will be concerned with computing the Gieseker wall.  

\subsubsection{Large volume limit}\label{sssec-largevolume} Let ${\bf v}\in K_{\num}(X)$ have positive rank, and consider stability conditions $\sigma_{\beta,\alpha}$ with $\beta<\overline\mu_{H,D}({\bf v})$ and $\alpha \gg 0$.  Maciocia \cite{Maciocia} shows that any $\sigma_{\beta,\alpha}$-semistable object of character ${\bf v}$ is a $\mu_{H}$-semistable (equivalently, $\mu_{H,D}$-semistable) sheaf.  Observe that if a $\mu_H$-semistable sheaf $E$ is $\sigma_{\beta,\alpha}$-(semi)stable for any $\beta,\alpha$ with $\beta < \overline\mu_{H,D}(E)$, then it is also $(H,D)$-Gieseker (semi)stable.  Indeed, if $F\subset E$ is an $(H,D)$-Gieseker stable subsheaf with $\mu_{H,D}(F) = \mu_{H,D}(E)$, then $$\nu_{\sigma_{\beta,\alpha}}(F) \underset{{\scriptscriptstyle (}-{\scriptscriptstyle )}}{<} \nu_{\sigma_{\beta,\alpha}} (E)$$
if and only if $$\overline \Delta_{H,D} (F) \geqpar \overline\Delta_{H,D}(E)$$ by our explicit formula for the slope $\nu_{\sigma_{\beta,\alpha}}$.  Thus for $\alpha \gg 0$ the moduli space $M_{\sigma_{\beta,\alpha}}({\bf v})$ coincides with $M_{H,D}({\bf v})$.  

If we only know that $\sigma_{\beta,\alpha}$ lies above the Gieseker wall, it a priori may be the case that $M_{\sigma_{\beta,\alpha}}({\bf v})$ is larger than $M_{H,D}({\bf v})$.  However, it is still true that every $(H,D)$-Gieseker semistable sheaf is $\sigma_{\beta,\alpha}$-semistable.  This is all we will need to apply the Positivity Lemma.

\subsubsection{The Positivity Lemma} Our main tool for constructing nef divisors on the moduli space is the Positivity Lemma of Bayer and Macr\`{i}. Let $\sigma=(Z, \cA)$ be a Bridgeland stability condition on $X$, ${\bf v} \in K_{\num}(X)$ and $S$ a proper algebraic space of finite type over $\C$. Let $\cE \in D^b(X \times S)$ be a flat family of $\sigma$-semistable objects. Denote the two projections on $X \times S$ by $p$ and $q$, respectively. Then Bayer and Macr\`{i} define a numerical class on $D_{\sigma, \cE} \in N^1(S)$ by setting $$D_{\sigma, \cE} \cdot C = -\Im\left(\frac{Z(p_*(\cE \otimes q^*\cO_C))}{Z({\bf v})} \right)$$ for every integral curve $C$ on $S$.

\begin{theorem}[{Positivity Lemma \cite[Lemma 3.3]{BayerMacri2}}]\label{lem-positivity}
The divisor $D_{\sigma, \mathcal{E}}$ is nef on $S$. A projective, integral curve $C\subset S$ satisfies $D_{\sigma, \mathcal{E}} \cdot C = 0$ if and only if objects parameterized by $C$ are generically $S$-equivalent with respect to $\sigma$.
\end{theorem}

From now on we assume that $M_{H,D}({\bf v})$ contains only stable sheaves.  Then $M_{H,D}({\bf v})$ admits a quasiuniversal family $\cE$ (see \cite[Theorem A.5]{Mukai} or \cite{HuybrechtsLehn} in the case of ordinary Gieseker stability).
 For $(\beta, \alpha)$ in the region bounded by the Gieseker wall and the vertical wall, $\mathcal{E}$ is a family of $\sigma_{\beta, \alpha}$-semistable  objects. Hence, the Positivity Lemma provides a nef divisor on $M_{H,D}({\bf v})$ associated to the Gieseker wall. By \cite[\S 4]{BayerMacri2}, the nef divisor may be identified with the  determinantal class corresponding to the unique vector $\alpha$ satisfying $$\Im(Z(-)) = \chi (\alpha \cdot -),$$ where $\chi$ denotes the Euler pairing. Thus, the construction of Bayer and Macr\`{i} can only give determinantal classes.  When the irregularity $q(X)=0$, it is natural to guess that the determinant line bundles span $\NS(M({\bf v}))$ if $\Delta({\bf v}) \gg 0$. This is known when $\rk({\bf v})=2$ by results of Jun Li \cite{JunLi}, but open in general.

\section{The destabilizing sequence and strong Bogomolov inequalities}\label{sec-destabilizing}

Fix divisors $(H,D)$ giving a slice of $\Stab(X)$.  Let $e>0$ be a generator of the subgroup $$H \cdot \Pic(X) \subset \Z.$$  We define the \emph{reduced slope} of a class ${\bf v}$ of positive rank by $$\tilde \mu_H({\bf v}) = \frac{H^2}{e} \mu_H({\bf v}) = \frac{c_1({\bf v})\cdot H}{r({\bf v})e}.$$ The set of reduced slopes of stable vector bundles on $X$ of rank at most $r$ is precisely the set of rational numbers with denominator at most $r$.  The reduced slope determines and is determined by the ordinary slope $\mu_H$ or any of the twisted slopes $\mu_{H,D}$, $\overline\mu_{H,D}$.

\subsection{Extremal characters} Consider a variable Chern character ${\bf v}$ where $r({\bf v})>0$ and $c_1({\bf v})$ are fixed but $\overline\Delta_{H,D}({\bf v})$ is variable, subject to the restriction that ${\bf v}$ is $(H,D)$-stable.  By Theorem \ref{thm-OGrady}, ${\bf v}$ will be $(H,D)$-stable so long as $\overline\Delta_{H,D}({\bf v})$ is sufficiently large and ${\bf v}$ is integral.  In this section, we describe the Gieseker wall for $M_{H,D}({\bf v})$ in the $(H,D)$-slice under the assumption that $\overline\Delta_{H,D}({\bf v})$ is sufficiently large.

First, we describe the numerical invariants of the destabilizing subobject.  
 
 \begin{definition}\label{def-extremal} An \emph{extremal character} ${\bf w}$ for ${\bf v}$ is any Chern character satisfying the following defining properties.
\begin{enumerate}[label=(E\arabic*)]
\item \label{cond-rankBound}We have $0<r({\bf w})\leq r({\bf v})$, and if $r({\bf w})=r({\bf v})$, then $c_1({\bf v})-c_1({\bf w})$ is effective.
\item \label{cond-slopeClose} We have $\tilde \mu_H({\bf w}) < \tilde \mu_H({\bf v})$, and $\tilde \mu_H({\bf w})$ is as close to $\tilde \mu_H({\bf v})$ as possible subject to \ref{cond-rankBound}.
\item \label{cond-stable} The moduli space $M_{H,D}({\bf w})$ is nonempty.
\item \label{cond-discriminantMinimal} The discriminant $\overline\Delta_{H,D}({\bf w})$ is as small as possible, subject to \ref{cond-rankBound}-\ref{cond-stable}.
\item \label{cond-rankMaximal} The rank $r({\bf w})$ is as large as possible, subject to \ref{cond-rankBound}-\ref{cond-discriminantMinimal}.
\end{enumerate}
\end{definition}

Note that conditions \ref{cond-rankBound} and \ref{cond-slopeClose} uniquely determine $\tilde \mu_H({\bf w})$.  Property \ref{cond-discriminantMinimal} uniquely determines $\overline\Delta_{H,D}({\bf w})$ (note that the Bogomolov inequality and the bound on the rank guarantee that a minimum actually exists), and property \ref{cond-rankMaximal} uniquely determines $r({\bf w})$.  Furthermore, the discriminant $\overline\Delta_{H,D}({\bf v})$ plays no role in the determination of ${\bf w}$.  Thus the triple $$(r({\bf w}),\tilde\mu_H({\bf w}),\overline\Delta_{H,D}({\bf w}))$$ is uniquely determined by $r({\bf v})$ and $c_1({\bf v})$; on the other hand, there may be several possible choices for $c_1({\bf w})$.  The requirement that $\overline\Delta_{H,D}({\bf w})$ is as small as possible may restrict which first Chern classes $c_1({\bf w})$ are permissible.

Our main result in this paper is the following.

\begin{theorem}\label{thm-main}
If  $\overline\Delta_{H,D}({\bf v})\gg 0$, then the Gieseker wall for $M_{H,D}({\bf v})$ in the $(H,D)$-slice is $W({\bf w},{\bf v})$.  There are curves in $M_{H,D}({\bf v})$ parameterizing sheaves which become $S$-equivalent along this wall.

\end{theorem}

There are two main steps to the proof of Theorem \ref{thm-main}.  First, we show that no actual wall for $M_{H,D}({\bf v})$ is larger than the wall $W({\bf w},{\bf v})$ given by ${\bf w}$.  This step is not too difficult; it will follow from a bound on higher rank walls and an asymptotic study of walls.  

Next, we prove this wall is the Gieseker wall and that the corresponding nef divisor lies on the boundary of the nef cone. Put ${\bf u} = {\bf v} - {\bf w}$.  We show that there are sheaves $F\in M_{H,D}({\bf w})$ and $Q\in M_{H,D}({\bf u})$ and curves in $\Ext^1(Q,F)$ such that the corresponding family of sheaves $E$ fitting as extensions $$0\to F\to E\to Q\to 0$$ are generically $(H,D)$-Gieseker stable and vary in moduli. Then the wall $W({\bf w},{\bf v})$ is the Gieseker wall since such sheaves $E$ are destabilized along it.  Furthermore, the corresponding curves in $M_{H,D}({\bf v})$ are orthogonal to the nef divisor given by the Gieseker wall, so the divisor is on the boundary of the nef cone.  This second part of the proof is fairly delicate, and primarily depends on computing the Gieseker wall for $M_{H,D}({\bf u})$ by induction on the rank.

\begin{remark}
Note that if ${\bf w}'$ is any character satisfying properties \ref{cond-rankBound}-\ref{cond-discriminantMinimal} in Definition \ref{def-extremal} (but not necessarily property \ref{cond-rankMaximal}), then the walls $W({\bf w},{\bf v})$ and $W({\bf w}',{\bf v})$ will coincide.  Property \ref{cond-rankMaximal} has been imposed to make the construction of orthogonal curves to the nef divisor as easy as possible.
\end{remark}

\subsection{The quotient character}\label{ssec-quotient} The definition of the extremal character ${\bf w}$ ensures that the moduli space $M_{H,D}({\bf w})$ is nonempty.  In the previous discussion we needed to know that the moduli space $M_{H,D}({\bf u})$ corresponding to the quotient character ${\bf u} = {\bf v} - {\bf w}$ is also nonempty.  We now address this point, and study ${\bf u}$ more closely.  There are two cases to consider, based on whether $r({\bf u})>0$ or $r({\bf u})=0$.

First assume $r({\bf u})>0$.   Note that $r({\bf u})$ and $\tilde \mu_H({\bf u})$ depend only on $r({\bf v})$ and $c_1({\bf v})$.  The class $c_1({\bf u})$ depends on the choice of $c_1({\bf w})$.  The relationship between $\overline\Delta_{H,D}({\bf u})$ and the other invariants is encoded in the identity \begin{align*}r({\bf v}) \overline\Delta_{H,D}({\bf v}) &= r({\bf w}) \overline\Delta_{H,D}({\bf w})+r({\bf u})\overline\Delta_{H,D}({\bf u})\\&\phantom{=}-\frac{r({\bf w})r({\bf u})}{2r({\bf v})}(\overline\mu_{H,D}({\bf w})-\overline\mu_{H,D}({\bf u}))^2.\end{align*} In particular, as $\overline\Delta_{H,D}({\bf v})\to \infty$ we find $\overline\Delta_{H,D}({\bf u})\to \infty$.  Therefore, if $\overline\Delta_{H,D}({\bf v})\gg 0$, Theorem \ref{thm-OGrady} applies to the moduli space $M_{H}({\bf u})$.  For instance, there are $\mu_{H,D}$-stable sheaves of character ${\bf u}$.

On the other hand, if $r({\bf u})  = 0$, then by \ref{cond-rankBound} we find that $c_1({\bf u})$ is effective. Let $C$ be an effective curve representing this class.  By \ref{cond-slopeClose}, the line bundle $\OO_X(C)$ has the smallest possible reduced slope $\tilde\mu_H(\OO_X(C))$ among all effective line bundles on $X$.  Therefore $C$ is reduced and irreducible, and the moduli space $M_{H,D}({\bf u})$ contains sheaves which are line bundles of the appropriate degree supported on $C$.

In either case, $M_{H,D}({\bf u})$ is nonempty and contains well-behaved points.

\subsection{Background on Farey sequences} The arguments in this paper rely on understanding the number theory which determines the slope $\tilde \mu_H({\bf w})$ of the exceptional character ${\bf w}$.  Recall that the \emph{(unrestricted) Farey sequence} $F_n$ of order $n$ consists of the ordered list of reduced fractions with denominator at most $n$.  For example,
$$F_6 = \left\{\ldots, -\frac{1}{6},\frac{0}{1}, \frac{1}{6}, \frac{1}{5}, \frac{1}{4}, \frac{1}{3},\frac{2}{5},\frac{1}{2},\frac{3}{5},\frac{2}{3},\frac{3}{4},\frac{4}{5},\frac{5}{6},\frac{1}{1},\frac{7}{6},\ldots\right\}$$
Thus the elements of $F_r$ are precisely the reduced slopes of stable vector bundles on $X$ of rank at most $r$.

Suppose that $$\frac{a}{b}<\frac{c}{d}$$ are \emph{Farey neighbors}, i.e. that they are adjacent terms in the Farey sequence $F_{\max\{b,d\}}$.  Then $$bc-ad=1.$$ The \emph{mediant} of two (reduced) rational numbers $\frac{a}{b}$ and $\frac{c}{d}$ is the rational number $$\frac{a+c}{b+d}.$$ If $\frac{a}{b}<\frac{c}{d}$ are Farey neighbors, then the mediant is already written in lowest terms.  Furthermore, the mediant is the \emph{unique} rational number in the interval $(\frac{a}{b},\frac{c}{d})$ with denominator at most $b+d$.  That is, the three terms $$\frac{a}{b},\frac{a+c}{b+d}, \frac{c}{d}$$ are adjacent in the Farey sequence $F_{b+d}$.

\begin{remark}\label{rem-Farey}
Here we explain how to use the Farey sequence to compute $\tilde \mu_H({\bf w})$ more explicitly.  Let $d = \tilde \mu_H(L)$ be the minimum reduced slope of an effective line bundle $L$ on $X$.  If $r({\bf v}) = 1$, then $\tilde \mu_H({\bf w}) = \tilde \mu_H({\bf v}) - d$.  If  $r({\bf v})\geq 2$, let $\alpha$ be the number in $F_{r({\bf v})}$ immediately preceding $\tilde \mu_H({\bf v})$.  If $\tilde \mu_{H}({\bf v})$ is not an integer, then the denominator of $\alpha$ is strictly less than $r({\bf v})$.  Therefore 
$$\tilde\mu_H({\bf w}) = \begin{cases} \alpha & \mbox{if $\tilde\mu_H({\bf v})\notin \Z$ or $d=1$} \\ \tilde\mu_H({\bf v})-\frac{1}{r({\bf v})-1} & \mbox{if $\tilde\mu_H({\bf v})\in \Z$ and $d>1$.}\end{cases}$$
\end{remark}

\subsection{Bogomolov inequalities} The extremal character ${\bf w}$ associated to ${\bf v}$ can be computed given the classification of stable Chern characters on $X$. For example, on $\PP^2$ it is easy to compute ${\bf w}$ from the Dr\'{e}zet-Le Potier classification (see \cite{CoskunHuizenga2} and \cite{LePotierLectures}). In particular, when $X= \P^2$, Theorem \ref{thm-main} specializes to the main theorem of \cite{CoskunHuizenga2}.

Conversely, suppose that $\Pic(X)=\Z H$ with $H$ effective.  In this case we can express Chern characters in terms of their rank $r$, slope $\mu = \tilde\mu_H$, and discriminant $\Delta = \Delta_{H,0}$. Fix a rank $r>0$ and slope $\mu\in \Q$ with $r\mu\in \Z$.  Let $\delta(r,\mu)$ be the minimal discriminant of a stable bundle $E$ such that $\mu(E) = \mu$ and $\rk(E) \leq r$.  Then the inequality $\Delta \geq \delta(r,\mu)$ is valid for any stable bundle with invariants $(r,\mu,\Delta)$, and typically improves the ordinary Bogomolov inequality.

\begin{corollary}\label{cor-Bogomolov}
Suppose $\Pic(X) = \Z H$ with $H$ effective.  Computing the Gieseker wall for all ${\bf v}$ with sufficiently large discriminant is equivalent to computing the function $\delta(r,\mu)$ for all $r>0$ and $\mu\in\Q$ with $r\mu\in \Z$.
\end{corollary}
\begin{proof}
If $\delta(r,\mu)$ is known, then it is straightforward to determine the character ${\bf w}$ from ${\bf v}$ using Remark \ref{rem-Farey}.  

Conversely, suppose the computation of the Gieseker wall is known.  Let $r>0$ and $\mu\in \Q$ with $c_1:=r\mu\in \Z$.  Write $\mu = \frac{b}{s}$ as a reduced fraction, so $r = ks$ for some integer $k>0$.  Define $\frac{b'}{s'}\in \Q$ by requiring $\frac{b}{s}<\frac{b'}{s'}$ to be Farey neighbors in the Farey sequence $F_r$. Then the mediant $$\mu':=\frac{b+b'}{s+s'}$$ has denominator $r':=s+s'>r$.  Let ${\bf v} = (r',\mu',\Delta')$.  Then the extremal Chern character ${\bf w}$ to ${\bf v}$ has slope $\mu$.  

What is the rank of ${\bf w}$?  Since ${\bf w}$ has slope $\mu$, we have $r({\bf w}) = ls$ for some $l>0$.  By definition, $s'\leq r$.  If $s' = r$, then since denominators of Farey neighbors are coprime we see that $\mu$ is an integer and $\delta(r,\mu) = 0$.  If instead $s'<r$, then $$(k+1)s = r+s>s'+s=r',$$ and thus $r({\bf w})\leq ks = r$.  It then follows that $\Delta({\bf w}) = \delta(r,\mu)$, computing $\delta(r,\mu)$.
\end{proof}

\begin{remark}
Note that the character ${\bf v}$ constructed in the proof of Corollary \ref{cor-Bogomolov} has coprime rank and first Chern class.  Thus the moduli space $M_H({\bf v})$ carries a universal family.  We conclude Corollary \ref{cor-BogIntro} holds as well.
\end{remark}

\section{Bounding the Gieseker wall}\label{sec-Gieseker}

In this section we show that if $\overline\Delta_{H,D}({\bf v})\gg 0$, then the Gieseker wall for $M_{H,D}({\bf v})$ in the $(H,D)$-slice is no larger than the wall $W:= W({\bf w},{\bf v})$ defined by the extremal Chern character ${\bf w}$ (see Definition \ref{def-extremal}).  Suppose $W'$ is a semicircular wall in the $(H,D)$-slice lying left of the vertical wall such that $W'$ is at least as large as $W$ and some $E\in M_{H,D}({\bf v})$ is destabilized along $W'$.  Let $\sigma_0$ be a stability condition on $W'$, and let $$0\to F' \to E \to Q'\to 0$$ be an exact sequence of $\sigma_0$-semistable objects of the same $\sigma_0$-slope which defines the wall $W'$.  

Let ${\bf w}' = \ch (F')$; then $W' = W({\bf w}', {\bf v})$.  We will show that if $\overline\Delta_{H,D}({\bf v})\gg 0$, then $\tilde \mu_{H}({\bf w'}) = \tilde \mu_H ({\bf w})$ and $\overline\Delta_{H,D}({\bf w}') = \overline\Delta_{H,D}({\bf w})$.  That is, the walls $W$ and $W'$ actually coincide, and the Gieseker wall is no larger than $W$.  

A now-standard argument gives some initial restrictions on $F'$.

\begin{lemma}\label{lem-slopeLess}
The object $F'$ is a nonzero torsion-free sheaf.  We have $\overline\mu_{H,D}({\bf w}')<\overline\mu_{H,D}({\bf v})$, and every Harder-Narasimhan factor of $F'$ has $(H,D)$-slope at most $\overline\mu_{H,D}({\bf v})$.
\end{lemma}
\begin{proof}
Fix a category $\cA_\beta$ such that some point $(\beta,\alpha)$ is on the wall $W'$.  Taking cohomology sheaves, since $\rH^{-1}(E)=0$ we find that $F'$ is a sheaf in $\cT_\beta$; it is nonzero since the wall $W^\dagger$ is not the whole slice.  Since $K:=\rH^{-1}(Q')$ and $E$ are torsion free, so is $F'$.  If $F'$ has an $(H,D)$-Gieseker stable subsheaf $F_1$ with $\overline \mu_{H,D}(F_1) > \overline\mu_{H,D}({\bf v})$, then $F_1$ is a subsheaf of $K$, which violates $K\in \cF_\beta$ since $E\in \cT_\beta$.  Finally, if $\overline\mu_{H,D}({\bf w}')=\overline\mu_{H,D}({\bf v})$, then $W'$ contains the vertical wall, which again is a contradiction.  Therefore $\overline\mu_{H,D}({\bf w}')<\overline\mu_{H,D}({\bf v})$.
\end{proof}

We next recall a lemma which first appeared in \cite{CoskunHuizenga2} for $\P^2$ and was later generalized in \cite{Betal}.  
\begin{lemma}[{\cite[Lemma 3.1]{Betal}}]\label{lem-higherRank}
With the notation and hypotheses of this section, if the map $F'\to E$ of sheaves is not injective, then the radius $\rho_{W'}$ of the wall $W'$ satisfies
 $$\rho_{W'}^2 \leq \frac{(\min\{r({\bf w}') -1,r({\bf v})\})^2}{2r({\bf w}')}\overline\Delta_{H,D}({\bf v}).$$ \end{lemma}

The lemma allows us to show the map of sheaves $F'\to E$ is injective once $\overline\Delta_{H,D}({\bf v})$ is sufficiently large.  This provides a restriction on the ranks of subobjects.

\begin{proposition}\label{prop-injective}
If $\overline\Delta_{H,D}({\bf v})\gg 0$, then the map $F' \to E$ of sheaves is injective.  In particular, $0 < r({\bf w'}) \leq r({\bf v})$.  Furthermore, in case $r({\bf w'}) = r({\bf v})$, the induced map on line bundles $\det F' \to \det E$ is an injection, and therefore $c_1({\bf v}) - c_1({\bf w'})$ is effective.
\end{proposition}
\begin{proof}
We compare the radius of $W$ with the bound on $\rho_{W'}^2$ in Lemma \ref{lem-higherRank}.  The center $(s_{W},0)$ and radius $\rho_{W}$ of $W$ satisfy
 \begin{align*}
 s_{W} &= \frac{\overline\mu_{H,D}({\bf v})+\overline\mu_{H,D}({\bf w})}{2}-\frac{\overline\Delta_{H,D}({\bf v})-\overline\Delta_{H,D}({\bf w})}{\overline\mu_{H,D}({\bf v})-\overline\mu_{H,D}({\bf w})}\\
 \rho_{W}^2 &= (\overline\mu_{H,D}({\bf v})-s_W)^2-2\overline\Delta_{H,D}({\bf v}).
 \end{align*}
 Therefore $\rho_{W}^2$ grows quadratically as a function of $\overline\Delta_{H,D}({\bf v})$. Let $$C=\max \left\{ \frac{(\min\{r'-1,r({\bf v})\})^2 }{2r'}:r'\in \N_{> 0}\right\}.$$ By Lemma \ref{lem-higherRank}, if the map $F'\to E$ is not injective, then $\rho_{W'}^2$ is bounded by $C\cdot\overline\Delta_{H,D}({\bf v})$.  Since $W'$ is at least as large as $W$, we conclude that if $\overline\Delta_{H,D}({\bf v})$ is sufficiently large, then $F'\to E$ is injective.
\end{proof}

Having restricted the rank $r({\bf w}')$, we next turn to restricting the slope $\tilde \mu_H({\bf w}')$ and discriminant $\overline\Delta_{H,D}({\bf w}')$.   We begin with a simple observation.

\begin{lemma}
Let $(x_W,0)$ be the right endpoint of the wall $W$, so that $x_W = s_{W} + \rho_{W}$.  Then $x_W$ is increasing as a function of $\overline\Delta_{H,D}({\bf v})$, and $$\lim_{\overline\Delta_{H,D}\to \infty} x_W = \overline\mu_{H,D}({\bf w}).$$
\end{lemma}
\begin{proof}
The walls $W = W({\bf w},{\bf v})$ are a family of numerical walls for ${\bf w}$, so they are all nested.  The formula for $s_{W}$ shows that the centers decrease (and tend to $-\infty$) as $\overline\Delta_{H,D}({\bf v})$ increases.  Correspondingly, the walls become larger and $x_W$ increases.  As the walls become arbitrarily large, they come arbitrarily close to the vertical wall $\beta = \overline\mu_{H,D}({\bf w})$, and the limit follows.
\end{proof}

We now complete the proof of the main theorem in this section.

\begin{theorem}\label{thm-GiesekerBound}
If $\overline\Delta_{H,D}({\bf v})\gg 0$, then $W'=W$.  Thus the Gieseker wall for $M_{H,D}({\bf v})$ is no larger than $W$.
\end{theorem}
\begin{proof}
Suppose $\overline\Delta_{H,D}({\bf v})$ is large enough that 
\begin{enumerate}
\item Proposition \ref{prop-injective} holds, and
\item $x_{W}$ is sufficiently close to $\overline\mu_{H,D}({\bf w})$ that no number in the interval $(x_W,\overline\mu_{H,D}({\bf w}))$ is the $\overline\mu_{H,D}$-slope of a Chern character of rank at most $r({\bf v})$.
\end{enumerate}
Since $W'$ is at least as large as $W$, the sheaf $F'$ lies in $\cT_{x_W}$. By Lemma \ref{lem-slopeLess}, we have $$\overline\mu_{H,D}({\bf w}') \in (x_W,\overline\mu_{H,D}({\bf v})).$$  More precisely, since $r({\bf w}')\leq r({\bf v})$ we actually have $$\overline\mu_{H,D}({\bf w}')\in [\overline\mu_{H,D}({\bf w}),\overline\mu_{H,D}({\bf v})).$$ Since we know that $c_1({\bf v})-c_1({\bf w}')$ is effective in case $r({\bf w}') = r({\bf v})$, we conclude from the definition of $\overline\mu_{H,D}({\bf w})$ that $\overline\mu_{H,D}({\bf w'}) = \overline\mu_{H,D}({\bf w})$.

The sheaf $F'$ is also $\overline\mu_{H,D}$-semistable, for if $F'$ has a quotient sheaf of smaller slope, then $F'$ is not in $\cT_{x_W}$ by construction.  Since $F'$ is $\sigma_0$-semistable, it is also $(H,D)$-Gieseker semistable by \S\ref{sssec-largevolume}  The formula for the center of a wall and the assumption that $W'$ is at least as large as $W$ implies $\overline\Delta_{H,D}({\bf w}') \leq \overline\Delta_{H,D}({\bf w})$. By the minimality of $\overline\Delta_{H,D}({\bf w})$, we conclude $\overline\Delta_{H,D}({\bf w}') = \overline\Delta_{H,D}({\bf w}).$
\end{proof}

\section{Nesting walls}\label{sec-nesting}

Let ${\bf w}$ denote the extremal Chern character from Definition \ref{def-extremal}. In the next section we will prove that $W  = W({\bf w},{\bf v})$ is actually the Gieseker wall for $M_{H,D}({\bf v})$ by producing $(H,D)$-Gieseker stable sheaves which are destabilized along $W$.  The main ingredient in this construction will be the inductive computation of the Gieseker wall for the moduli space $M_{H,D}({\bf u})$ corresponding to the quotients, which we address here.  Recall that ${\bf u} = {\bf v}-{\bf w}$.

\begin{proposition}\label{prop-nesting}
Assume Theorem \ref{thm-main} holds for characters of (positive) rank less than $r({\bf v})$. If $\overline\Delta_{H,D}({\bf v})\gg 0$, then the Gieseker wall for $M_{H,D}({\bf u})$ in the $(H,D)$-slice is nested properly inside $W$.  

(If the rank $r({\bf u})$ is zero, it may happen that every $E\in M_{H,D}({\bf u})$ is semistable everywhere in the $(H,D)$-slice.  In this case we consider the Gieseker wall to be empty and the result is vacuous.) 
\end{proposition}

The proof of Proposition \ref{prop-nesting} is different based on whether $r({\bf u})>0$ or $r({\bf u})=0$. We treat the more interesting positive rank case first.

\subsection{Positive rank quotients} Throughout this subsection we assume $r({\bf u})>0$, i.e. $0< r({\bf w})<r({\bf v})$.  In particular, $r({\bf v})\geq 2$.  We write the inequality $\tilde \mu_H({\bf w}) < \tilde \mu_H({\bf v})< \tilde \mu_H({\bf u})$ as $$\frac{a'}{r'}< \frac{a}{r} < \frac{a''}{r''},$$ where the denominators are the ranks of the corresponding characters. We begin with a couple useful lemmas.  Write the above fractions in lowest terms as $$\frac{b'}{s'} < \frac{b}{s} < \frac{b''}{s''},$$ with positive denominators.
\begin{lemma}\label{lem-farey}
The fractions $\frac{b'}{s'}$ and $\frac{b}{s}$ are Farey neighbors.
\end{lemma}
\begin{proof}
This follows immediately from Remark \ref{rem-Farey} since $r({\bf v})\geq 2$.
\end{proof}

  \begin{lemma}\label{lem-reducedFraction}
 At least one of the fractions $\frac{a'}{r'}$ or $\frac{a}{r}$ is already written in lowest terms.
 \end{lemma}
 \begin{proof}
 Suppose not.  Then $r\geq 2s$ and $r'\geq 2s'$.  The denominator $s'+s$ of the mediant $$\frac{b'}{s'}< \frac{b'+b}{s'+s} < \frac{b}{s}$$ satisfies $$s'+s \leq \frac{1}{2}(r'+r) \leq r.$$  If $s'+s<r$, then condition \ref{cond-slopeClose} defining $\tilde \mu_H({\bf w})$ is violated.  On the other hand, if $s'+s=r$, then $r=r'$, contrary to our assumption.
 \end{proof}
 
We now relate an extremal character for ${\bf u}$ to ${\bf w}$.

\begin{lemma}\label{lem-quotientInvariants}
Let ${\bf w}^\dagger$ be an extremal character for ${\bf u}$. Then $\tilde \mu_H({\bf w}^\dagger) \leq \tilde \mu_H({\bf w})$, and in case of equality we have $\overline\Delta_{H,D}({\bf w}^\dagger) > \overline\Delta_{H,D}({\bf w})$.
\end{lemma}
\begin{proof}
First we show that $\tilde \mu_H({\bf w}^\dagger) \leq \tilde \mu_H ({\bf w})$.  We write $\tilde \mu_H({\bf w}^\dagger) = \frac{a^\dagger}{r^\dagger}$, and have $\frac{a^\dagger}{r^\dagger} < \frac{a''}{r''}$ and $r^\dagger \leq r''$.  It suffices to show that $\frac{a^\dagger}{r^\dagger} \notin (\frac{a'}{r'},\frac{a''}{r''})$.    We consider three different cases, depending on the relationship between $\frac{a^\dagger}{r^\dagger}$ and $\frac{a}{r}$.

\emph{Case 1:} Suppose $\frac{a^\dagger}{r^\dagger} \in (\frac{a'}{r'},\frac{a}{r})$.  Since $r^\dagger \leq r'' < r$, this contradicts the definition of $\tilde \mu_H({\bf w})$.

\emph{Case 2:} Suppose $\frac{a^\dagger}{r^\dagger} \in (\frac{a}{r},\frac{a''}{r''})$.  Let \begin{align*} a^{\ddagger}&=a-a^\dagger\\ r^{\ddagger} &= r-r^\dagger.
\end{align*} Then $\frac{a^{\ddagger}}{r^{\ddagger}}\in (\frac{a'}{r'},\frac{a}{r})$, again contradicting the definition of $\tilde \mu_H({\bf w})$.  Indeed, to prove this, we
can view $\frac{a}{r}$ as a weighted mean in two ways:
$$\frac{a}{r} = \frac{r' \frac{a'}{r'}+r''\frac{a''}{r''}}{r'+r''} = \frac{r^\ddagger\frac{a^\ddagger}{r^\ddagger}+r^\dagger\frac{a^\dagger}{r^\dagger}}{r^\ddagger+r^\dagger}.$$ Since $\frac{a^\dagger}{r^\dagger}$ is closer to $\frac{a}{r}$ than $\frac{a''}{r''}$ is and the weight $\frac{r^\dagger}{r}$ on $\frac{a^\dagger}{r^\dagger}$ in the second mean is smaller than the weight $\frac{r''}{r}$ on $\frac{a''}{r''}$ in the first mean, it follows that $\frac{a^\ddagger}{r^\ddagger}\in (\frac{a'}{r'},\frac{a}{r})$.  

\emph{Case 3:} Suppose $\frac{a^\dagger}{r^\dagger} = \frac{a}{r}$.  Since $r^\dagger <r$, we find that the fraction $\frac{a}{r}$ is not already reduced, and therefore $a'=b'$ and $r'=s'$ by Lemma \ref{lem-reducedFraction}.  We must have $s+s' \geq r$, for otherwise the mediant of $\frac{b'}{s'}$ and $\frac{b}{s}$ lies in $(\frac{a'}{r'},\frac{a}{r})$ and has denominator less than $r$, contradicting the definition of $\tilde\mu_H({\bf w})$.  Thus $$r'' = r-s'\leq s.$$ If this inequality is strict, then $r^\dagger\leq r''$ gives a contradiction.

Suppose instead that $r'' = s$ and $s+s'=r$.  Since $\frac{b'}{s'}$ and $\frac{b}{s}$ are  Farey neighbors by Lemma \ref{lem-farey}, their mediant $$\frac{b'+b}{s'+s}$$ is written in lowest terms and has denominator $r$.  Then $\frac{b'+b}{r}$ and $\frac{b}{s}$ are consecutive terms in the Farey sequence of order $r$, so $\gcd(r,s)=1$.  Therefore $r'' = s = 1$, the slope $\tilde \mu_H({\bf v})=b$ is an integer, and $\tilde \mu_H({\bf u}) = b+1$.  Since $\tilde \mu_H({\bf w}) \neq b-\frac{1}{r}$, there is no effective line bundle on $X$ of reduced slope $1$.  This implies $\tilde \mu_H({\bf w}^\dagger) \neq b$, a contradiction.

This completes the proof that $\tilde \mu_H({\bf w}^\dagger) \leq \tilde \mu_H({\bf w})$. Suppose $\tilde \mu_H({\bf w}^\dagger) = \tilde \mu_H({\bf w})$.  Let ${\bf w}_0$ be any character satisfying conditions \ref{cond-rankBound}-\ref{cond-discriminantMinimal} of the definition of an extremal character for ${\bf v}$, but such that the rank $r_0'$ of ${\bf w}_0$ is as \emph{small} as possible.  We must have $r'+ r_0' \geq r$ by condition \ref{cond-rankMaximal} in the definition of $\bf w$, since otherwise ${\bf w}+{\bf w}_0$ would satisfy \ref{cond-rankBound}-\ref{cond-discriminantMinimal} but have larger rank.  Therefore $$r^\dagger \leq r'' = r-r' \leq r_0'.$$  If the strict inequality $r^\dagger < r_0'$ holds, then by the definition of ${\bf w}_0$ we have $\overline\Delta_{H,D}({\bf w}^\dagger) >\overline\Delta_{H,D}({\bf w})$.

Instead assume $r^\dagger = r'' = r_0'$.  In this case we have $\overline\Delta_{H,D}({\bf w}^\dagger) \geq \overline\Delta_{H,D}({\bf w})$, so assume $\overline\Delta_{H,D}({\bf w}^\dagger) = \overline\Delta_{H,D}({\bf w})$.  Then ${\bf w}_0$ and ${\bf w}^\dagger$ have the same invariants, except possibly their first Chern classes are different.  But then ${\bf w}+{\bf w}^\dagger$ satisfies conditions \ref{cond-rankBound}-\ref{cond-discriminantMinimal} in the definition of ${\bf w}$, since $$c_1({\bf v})-(c_1({\bf w})+c_1({\bf w}^\dagger)) = c_1({\bf u})-c_1({\bf w}^\dagger)$$ and $c_1({\bf u})-c_1({\bf w}^\dagger)$ is effective.  This contradicts the condition \ref{cond-rankMaximal} in the definition of $\bf w$.  Therefore $\overline\Delta_{H,D}({\bf w}^\dagger) > \overline\Delta_{H,D}({\bf w})$ holds in this case as well.
\end{proof}

The lemma immediately allows us to treat the positive rank quotient case of Proposition \ref{prop-nesting}.

\begin{proof}[Proof of Proposition \ref{prop-nesting} when $r({\bf u})>0$]
By Lemma \ref{lem-quotientInvariants}, we have $\overline\mu_{H,D}({\bf w}^\dagger) \leq \overline\mu_{H,D}({\bf w})$, and in case of equality $\overline\Delta_{H,D}({\bf w}^\dagger) > \overline\Delta_{H,D}({\bf w})$.  We must compare the walls $W = W({\bf w},{\bf v}) = W({\bf w},{\bf u})$ and $W^\dagger := W({\bf w}^\dagger,{\bf u})$; note that both walls are numerical walls for ${\bf u}$, so they are disjoint, and it suffices to compare their right endpoints.  Recall that $\overline\Delta_{H,D}({\bf u})$ is an increasing function of $\overline\Delta_{H,D}({\bf v})$, and $\overline\Delta_{H,D}({\bf u})\to \infty$ as $\overline\Delta_{H,D}({\bf v})\to \infty$.

First suppose $\overline\mu_{H,D}({\bf w}^\dagger) < \overline\mu_{H,D}({\bf w})$.  Then the right endpoints $x_W$ and $x_{W^\dagger}$ of $W$ and $W^\dagger$ are increasing functions of $\overline\Delta_{H,D}({\bf v})$.  Furthermore, as $\overline\Delta_{H,D}({\bf v})\to \infty$, we have $x_W \to \overline\mu_{H,D}({\bf w})$ and $x_{W^\dagger} \to \overline\mu_{H,D}({\bf w}^\dagger)$.  Therefore if $\overline\Delta_{H,D}({\bf v})$ is sufficiently large, $x_{W^\dagger} < x_W$, and $W^\dagger$ is nested in $W$.

Next suppose $\overline\mu_{H,D}({\bf w}^\dagger) = \overline\mu_{H,D}({\bf w})$ and $\overline\Delta_{H,D}({\bf w}^\dagger)  > \overline\Delta_{H,D}({\bf w})$.  Comparing the formulas for the centers of $W({\bf w},{\bf u})$ and $W({\bf w}^\dagger,{\bf u})$ immediately proves the result; we don't even need to increase $\overline\Delta_{H,D}({\bf v})$.
\end{proof}
 
\subsection{Rank zero quotients}  In this case things are considerably easier.  

\begin{proof}[Proof of Proposition \ref{prop-nesting} when $r({\bf u})=0$] Record the character ${\bf u}$ in terms of its first Chern class $c_1({\bf u})$ and Euler characteristic $\chi({\bf u})$, which depends on $\overline\Delta_{H,D}({\bf v})$.  By the construction of ${\bf w}$, $c_1({\bf u})$ is effective, so the moduli spaces $M_{H,D}({\bf u})$ are all nonempty.  Since tensoring by $\OO_X(H)$ gives an isomorphism $$M_{H,D}(c_1({\bf u}),\chi({\bf u})) \cong M_{H,D}(c_1({\bf u}), \chi({\bf u})+H\cdot c_1({\bf u})),$$ there are only finitely many isomorphism types of spaces $M_{H,D}({\bf u})$ as $\overline\Delta_{H,D}({\bf v})$ varies.  Tensoring by $\OO_X(H)$ also preserves the radius of the Gieseker wall, assuming the Gieseker wall of either space is nonempty.  Therefore, there is a universal bound on the radii of the Gieseker walls of the spaces $M_{H,D}({\bf u})$.  

Recall that the numerical walls for ${\bf u}$ are nested semicircles with a common center that foliate the entire $(H,D)$-slice \cite{CoskunHuizenga,Maciocia}.  Since $W = W({\bf w},{\bf v}) = W({\bf w},{\bf u})$ is also a numerical wall for for ${\bf u}$, a numerical wall for ${\bf u}$ is nested inside $W$ if and only if its radius is smaller than the radius of $W$.  Since $W$ is arbitrarily large for $\overline\Delta_{H,D}({\bf v})\gg 0$, this completes the proof.
\end{proof}

\section{Orthogonal curves}\label{sec-curves}

In this section we prove that $W$ is actually the Gieseker wall by producing curves of objects in $M_{H,D}({\bf v})$ which are destabilized along $W$.  If $M_{H,D}({\bf v})$ contains only stable sheaves, our curves will furthermore be orthogonal to the nef divisor given by $W$.  We first recall some algebraic preliminaries.

\subsection{Extensions}\label{ssec-extensions} The basis for our construction of stable sheaves is the following mild generalization of \cite[Lemma 6.9]{BayerMacri3}.  Recall that a simple object in an abelian category is an object with no proper subobjects, and a semisimple object is a (finite) direct sum of simple objects.  In what follows we write $$A = \bigoplus_i A_i^{n_i}$$ with the $A_i$ simple and nonisomorphic. Then every  subobject or quotient object of $A$ is isomorphic to an object $\bigoplus A_i^{m_i}$ for some integers $m_i$ with $0\leq m_i \leq n_i$.  In particular, every quotient of $A$ is also a subobject of $A$.

\begin{lemma}\label{lem-subobjects}
Let $\cA$ be an abelian category, and let $A,B\in \cA$ with $A$ semisimple and $B$ simple.  If $E$ is any extension of the form $$0\to A\to E \to B \to 0$$ with $\Hom(E,A)=0$, then any subobject of $E$ is a subobject of $A$.
\end{lemma}
\begin{proof}
Let $S\subset E$ be a subobject.  Consider the composition $\phi: S\to E \to B$.  Since $B$ is simple, $\phi$ is either surjective or zero.  If $\phi$ is zero, then $S$ is a subobject of $A$.

Suppose instead that $\phi$ is surjective; in this case we will obtain a contradiction.  Let $C$ be the cokernel of the inclusion $$0\to S \to E \to C \to 0.$$ Then the composition $A\to E \to C$ is surjective, so $C$ is a quotient of $A$.  Since $A$ is semisimple, $C$ is also isomorphic to a subobject of $A$, but this contradicts $\Hom(E,A)=0$.
\end{proof}

The next lemma gives a criterion for the vanishing $\Hom(E,A)=0$ needed to apply Lemma \ref{lem-subobjects}.  Recall that by Schur's lemma, if $A,B$ are simple objects in a $\C$-linear abelian category, then $\Hom(A,B)=0$ unless $A,B$ are isomorphic, and $\Hom(A,A)\cong \C$.

\begin{lemma}\label{lem-noHoms}
Let $\cA$ be a $\C$-linear abelian category, and let $A,B\in \cA$ with $A$ semisimple and $B$ simple.  Assume $B$ is not a simple factor of $A$.  Consider an extension $$0\to A\to E\to B\to 0$$ given by an extension class $e\in \Ext^1(B,A) \cong \bigoplus_i \Ext^1(B,A_i)^{n_i}$.  For each $i$, write $e_{i,1},\ldots,e_{i,n_i}$ for the $n_i$ components of $e$ under this isomorphism.  Then $\Hom(E,A) = 0$ if and only if $e_{i,1},\ldots,e_{i,n_i}$ are linearly independent in $\Ext^1(B,A_i)$ for all $i$.

In particular, if $E$ is a general extension as above, then $\Hom(E,A)=0$ if and only if $\ext^1(B,A_i) \geq n_i$ for all $i$.
\end{lemma}
\begin{proof}
Observe that $\Hom(E,A)=0$ if and only if $\Hom(E,A_i)=0$ for all $i$.  Applying $\Hom(-,A_i)$ to the sequence defining $E$ and using that $B\not\cong A_i$, we get an exact sequence $$0\to \Hom(E,A_i) \to \Hom(A_i,A_i)^{\oplus n_i} \to \Ext^1(B,A_i).$$ The map $\C^{n_i}\cong \Hom(A_i,A_i)^{\oplus n_i}\to \Ext^1(B,A_i)$ carries the identity map in the $j$th component of $\Hom(A_i,A_i)^{\oplus n_i}$ to $e_{i,j}$, so this map is injective if and only if the $e_{i,j}$ are linearly independent. 
\end{proof}

Finally, we study when two extensions are isomorphic.

\begin{lemma}\label{lem-groupAction}
Let $\cA$ be a $\C$-linear abelian category, and let $A,B\in \cA$ with $A$ semisimple, $B$ simple, and $B$ not a simple factor of $A$.  If \begin{align*} 0 \to A \to E\hphantom{'} \to B \to 0 \\ 0\to A\to E'\to B\to 0\end{align*} are two extensions of $B$ by $A$, then any isomorphism $E\to E'$ is induced by an automorphism of $A$.  

Therefore, if $\ext^1(B,A) > \dim \Aut A = \sum_i n_i^2$ then two general extensions $E,E'$ as above are non-isomorphic.
\end{lemma}
\begin{proof}
Since $\Hom(B,B)\cong \C$ and $\Hom(A,B)=0$, we have $\Hom(E,B)=\C$.  Thus up to scale the maps $E\to B$ and $E'\to B$ are canonically determined by $E,E'$.  Their kernels are therefore identified under the isomorphism $E\to E'$.
\end{proof}

\subsection{Construction of curves}  We now bring together the results of Sections \ref{sec-Gieseker}, \ref{sec-nesting} and \ref{ssec-extensions} to prove our main result.

\begin{theorem}
If $\overline\Delta_{H,D}({\bf v})\gg 0$, then $W$ is the Gieseker wall for $M_{H,D}({\bf v})$.  Furthermore, there are curves in $M_{H,D}({\bf v})$ parameterizing non-isomorphic $(H,D)$-twisted Gieseker stable sheaves that become $S$-equivalent for Bridgeland stability conditions along $W$. 
\end{theorem}
\begin{proof}
Let $\sigma_0 = (Z_0,\cA_0)$ be a stability condition on $W$.  Choose a polystable sheaf $$F = \bigoplus_i F_i^{n_i}\in M_{H,D}({\bf w}).$$  Since $W$ can be made arbitrarily large by increasing $\overline\Delta_{H,D}({\bf v})$, we may assume that every stable factor of $F$ is $\sigma_0$-stable.  By Theorem \ref{thm-OGrady}, we may increase $\overline\Delta_{H,D}({\bf v})$ so that there are stable sheaves in $M_{H,D}({\bf u})$.  We let $Q$ be such a stable sheaf (if ${\bf u}$ has rank $0$, we additionally assume $Q$ is sufficiently nice; see Step 2 below).  Increasing $\overline\Delta_{H,D}({\bf v})$, Proposition \ref{prop-nesting} shows that $Q$ is actually $\sigma_0$-stable.  It is clear that $Q$ is not one of the stable factors of $F$.
Increasing $\overline\Delta_{H,D}({\bf v})$ decreases the Euler characteristics $\chi({\bf u},F_i)$.
Thus, if $\overline\Delta_{H,D}({\bf v})$ is sufficiently large we will have $\chi(Q,F_i) \leq -n_i$ and $\chi(Q,F) < - \sum_i n_i^2$.

Let $\cP\subset \cA_0$ be the full subcategory of $\sigma_0$-semistable objects with the same $\sigma_0$-slope as $F$ and $Q$.  Then $F$ is a semisimple object of $\cP$ and $Q$ is a simple object of $\cP$.  If $E$ is a general extension of the form $$0\to F\to E \to Q\to 0,$$ then by Lemma \ref{lem-noHoms} we have $\Hom(E,F)=0$.  Furthermore, by Lemma \ref{lem-groupAction} we can find curves in $\Ext^1(Q,F)$ such that two general parameterized objects $E$ are nonisomorphic.  To complete the proof, we prove that $E$ is $(H,D)$-Gieseker stable.  

\emph{Step 1: if $\sigma_+$ is a stability condition just above $W$, then $E$ is $\sigma_+$-stable}.  Suppose $F'\subset E$ destabilizes $E$ with respect to $\sigma_+$, so $\mu_{\sigma_+}(F')\geq \mu_{\sigma_+}(E)$.  Since $E$ is $\sigma_0$-semistable, we have $\mu_{\sigma_0}(F') = \mu_{\sigma_0}(E)$, and thus by Lemma \ref{lem-subobjects} $F'$ is a subobject of $F$ in $\cP$.  But then $\mu_{\sigma_+}(F') = \mu_{\sigma_+}(F)<\mu_{\sigma_+}(E)$, a contradiction.

\emph{Step 2: $E$ is torsion-free.}  If ${\bf u}$ has positive rank this is trivial, so assume that $r({\bf u})=0$.  By the discussion in \S\ref{ssec-quotient}, we may assume $Q$ is a line bundle $L$ supported on a reduced and irreducible curve $C$.  Suppose $E$ has a nonzero torsion subsheaf $T$, and let $E' = E/T$.  Since $F$ is torsion-free, $T$ must be a subsheaf of $Q$.  Since $Q$ has pure dimension $1$, $T$ must be another line bundle $L'$ supported on $C$.  Then $c_1(T) = c_1({\bf u})$, so $$c_1(E') = c_1({\bf v}) - c_1({\bf u}) = c_1({\bf w}).$$ Furthermore, the composition $F\to E\to E'$ is injective.  Since the stable factors of $F$ have minimal discriminant, this is only possible if $F\to E'$ is an isomorphism.  But then the composition $E\to E'\to F$ with the inverse isomorphism gives a nontrivial homomorphism $E\to F$, which is a contradiction.  Therefore $E$ is torsion-free.

\emph{Step 3: $E$ is $\mu_{H,D}$-semistable.} Suppose that $E\to C$ is a $\mu_{H,D}$-stable quotient of $E$ with $\mu_{H,D}(C) < \mu_{H,D}(E)$ and $r(C)<r(E)$.  By the definition of ${\bf w}$, we have $\mu_{H,D}(C)\leq \mu_{H,D}(F)$.  If $\mu_{H,D}(C)<\mu_{H,D}(F)$, then the composition $F\to E \to C$ is $0$, so induces a map $Q\to C$ which must be $0$ since $\mu_{H,D}(C)<\mu_{H,D}(Q)$; thus $C$ is zero, a contradiction.  If instead $\mu_{H,D}(C) = \mu_{H,D}(F)$, since the stable factors of $F$ have minimal discriminant the composition $F\to E\to C$ is either $0$ or identifies $C$ with one of the stable factors $F_i$ of $F$.  In the first case we conclude as before, and in the second we obtain a nontrivial homomorphism $E\to F$, which again is a contradiction.  Thus $E$ is $\mu_{H,D}$-semistable.
 
Finally, since $E$ is $\mu_{H,D}$-semistable and $\sigma_+$-stable, it is $(H,D)$-Gieseker stable by the discussion in \S\ref{sssec-largevolume}.  This completes the proof. 
\end{proof}

\section{Examples}\label{sec-examples}

In this section, we give applications and examples of our general theory. We discuss the nef cones of certain moduli spaces of vector bundles of rank $2$ on several classes of  surfaces.  Let $(X,H)$ be a polarized surface and consider  the vector ${\bf v}$ with $r({\bf v}) = 2$, fixed $\ch_1({\bf v})$, and variable $\ch_2({\bf v})\ll 0$, so that $\overline \Delta_{H,D}({\bf v}) \gg 0$.  In the cases we consider, $\mu_H$-semistability and $\mu_H$-stability will coincide, so the moduli space $M_H({\bf v})$ carries a quasiuniversal family.  Additionally, we will have $q(X) = 0$.  

Under these assumptions, by O'Grady's theorem \cite{OGrady} the moduli space $M_H({\bf v})$ is irreducible.  Writing ${\bf v}^\perp$ for the orthogonal complement of ${\bf v}$ in $K_{\num}(X)$ with respect to the Euler pairing $({\bf v},{\bf w}) = \chi({\bf v}\te {\bf w})$, the Donaldson homomorphism $$\lambda : {\bf v}^\perp \to N^1(M_H({\bf v}))$$ is an isomorphism by a theorem of Jun Li \cite{JunLi}.  Thus we can specify divisor classes by giving elements of ${\bf v}^\perp$.  By \cite[Proposition 3.8]{Betal}, if $\sigma$ is a stability condition on a wall $W$ in the $(H,D)$-slice of stability conditions, then the Bayer-Macr\`i divisor class associated to $\sigma$ corresponds to  a multiple of the class $$(-1,s_WH+D,m)\in {\bf v}^\perp,$$ where the second Chern character $m$ is determined by the requirement $(-1,s_WH+D,m)\in {\bf v}^\perp$.  

Additionally, when $\Pic(X) \cong \ZZ H$, one extremal ray of the nef cone corresponds to the Jun Li morphism to the Donaldson-Uhlenbeck-Yau compactification \cite[\S 8]{HuybrechtsLehn} and is given by $\lambda(0, H, n)$, where $n$ is chosen by the requirement that $(0, H, n) \in {\bf v}^{\perp}$. Hence, we only need to compute the other extremal ray.

\subsection{Surfaces in $\P^3$}
Let $X$ be a very general surface of degree $d \geq 4$ in $\P^3$. By the Noether-Lefschetz theorem, $\Pic(X) \cong \ZZ H$, where $H$ is the class of a hyperplane section of $X$. Let ${\bf v}$ be the Chern character with $r({\bf v})=2$, $c_1({\bf v}) = H$ and variable second Chern character $\ch_2({\bf v})$. The reduced slope is $\tilde{\mu}_H({\bf v})=\frac{1}{2}$. Since there are no line bundles with this reduced slope, every $\mu_H$-semistable sheaf of character ${\bf v}$ is $\mu_H$-stable.  We use the twisting divisor $D=0$, and omit it from our notation.

The extremal  character $\bf{w}$ must have reduced slope $\tilde{\mu}_H({\bf w})=0$ and rank at most $2$. The line bundle $\OO_X$ has this slope.  By the Bogomolov inequality, we conclude that the extremal character ${\bf w}$ is the character of $\OO_X \oplus \OO_X$. If $\Delta_H ({\bf v})$ is sufficiently large, then the Gieseker wall is given by $W({\bf v}, {\bf w})$.  A computation shows that the wall $W(\bf{v}, \bf{w})$ has center
$$s_W= - \frac{K_X \cdot H}{2H^2} + \frac{\ch_2({\bf v})}{H^2} = - \frac{d-4}{2} + \frac{\ch_2({\bf v})}{d}.$$  The interesting extremal ray of the nef cone corresponds to the class $(-1,  s_W H, m)\in {\bf v}^\perp$. 

\subsection{Double covers of $\PP^2$}
Let $X$ be the cyclic double cover of $\P^2$ branched along a very general curve of degree $2d \geq 6$. Let $H$ be the pullback of $\OO_{\P^2}(1)$. By the Noether-Lefschetz theorem (see for example \cite{RavindraSrinivas}), $\Pic(X) \cong \ZZ H$. Let ${\bf v}$ be the Chern character with $r({\bf v})=2$, $c_1({\bf v}) = H$ and variable second Chern character $\ch_2$. The reduced slope is $\tilde{\mu}_H({\bf v})=\frac{1}{2}$. Since there are no line bundles with this reduced slope, every $\mu_H$-semistable sheaf of character ${\bf v}$ is $\mu_H$-stable.  We again use the twisting divisor $D=0$.  

The extremal  character $\bf{w}$ must have reduced slope $\tilde{\mu}_H({\bf w})=0$ and rank at most $2$. The line bundle $\OO_X$ has this slope.  By the Bogomolov inequality, we conclude that the extremal character is the character of $\OO_X \oplus \OO_X$. If $\Delta_H ({\bf v})$ is sufficiently large, then the Gieseker wall is given by $W({\bf v}, {\bf w})$.  A computation shows that the wall $W(\bf{v}, \bf{w})$ has center
$$s_W= - \frac{K_X \cdot H}{2H^2} + \frac{\ch_2({\bf v})}{H^2} = - \frac{d-3}{2} + \frac{\ch_2}{2}.$$ One edge of the nef cone corresponds to a class $$(-1,s_W, H,m)\in {\bf v}^\perp,$$ and the other edge corresponds to the Donaldson-Uhlenbeck-Yau compactification.

\subsection{The quadric $\P^1\times \P^1$} 
Let $X = \P^1\times \P^1$.  We write classes in $N^1(X)_\Q \cong \Q H_1\oplus \Q H_2$ as $(a,b)$, where $H_1$ and $H_2$ are the two fibers on $X$.  Fix the polarization $H=(1,1)$, and define a family $D_t = (t,-t)$ of twisting divisors orthogonal to $H$.  We consider the vector ${\bf v}$ with $r({\bf v})=2$, $c_1({\bf v}) = (1,0)$, and variable $\ch_2({\bf v})\ll 0$.  Since we vary the twisting divisor in this section, it is preferable to view $\ch_2$ as varying instead of the twisted discriminant as varying, since the latter depends on the particular twist.

Observe that every $\mu_H$-semistable sheaf of character ${\bf v}$ is $\mu_H$-stable, since $\tilde \mu_H({\bf v}) = \frac{1}{2}$ and no line bundle has this reduced slope.  Therefore, the twisted moduli spaces $M_{H,D_t}({\bf v})$ parameterize the same objects as the ordinary Gieseker space $M_H({\bf v})$, and $M_H({\bf v})$ can be realized as the large-volume limit in the $(H,D_t)$-slice for any $t\in \Q$.  For any $t\in \Q$, Corollary \ref{cor-main} allows us to determine a boundary nef divisor on $M_H({\bf v})$ so long as $\ch_2({\bf v})\ll 0$; however, the required bound on $\ch_2({\bf v})$ depends on the particular time $t$.  As $\ch_2({\bf v})$ becomes more negative, we will find that the structure of the nef cone becomes increasingly complicated; for example, while any fixed space $M_H({\bf v})$ is a Mori dream space and thus has a finite polyhedral nef cone \cite{Ryan}, the number of extremal rays of this cone becomes arbitrarily large as $\ch_2({\bf v})$ decreases. 

Let ${\bf w}_t$ denote an extremal character for ${\bf v}$ in the $(H,D_t)$-slice.  In the next result we compute ${\bf w}_t$ in terms of $t$.

\begin{proposition}
Let $t\in \Q$ and let $n\in \Z$ be an integer closest to $t$.  
\begin{enumerate}
\item If $n\neq 0$, then the character of $\OO_X(n,-n)$ is an extremal character for ${\bf v}$ in the $(H,D_t)$-slice.
\item If $n=0$, then the character of $\OO_X^{\oplus 2}$ is an extremal character for ${\bf v}$ in the $(H,D_t)$-slice.
\end{enumerate}
Any extremal character for ${\bf v}$ in the $(H,D_t)$-slice is obtained in this way.  In particular, the extremal character is uniquely determined  if $t$ is not a half-integer, and there are two choices of extremal character if $t$ is a half-integer.\end{proposition}
\begin{proof}
Let ${\bf w}_t$ denote an extremal character for ${\bf v}$ in the $(H,D_t)$-slice.  Since $\tilde \mu_H({\bf w}_t)=0$, the first Chern class satisfies $c_1({\bf w}_t) = (n,-n)$ for some $n\in \Z$.  If ${\bf w}_t$ has rank $2$, then $$c_1({\bf v}) - c_1({\bf w}_t) = (1,0)-(n,-n) = (1-n,n)$$ is effective, so either $n=0$ or $n=1$.  In all other cases, ${\bf w}_t$ has rank $1$.

Next we minimize the discriminant of ${\bf w}_t$, viewing $c_1({\bf w}_t) = (n,-n)$ and $r({\bf w}_t)$ as fixed.  First recall the ordinary discriminant $\Delta$ and twisted discriminant $\overline\Delta_{H,D_t}$ are $$\Delta = \frac{1}{2} \frac{c_1^2}{r^2} - \frac{\ch_2}{r} \qquad\mbox{and}\qquad \overline\Delta_{H,D_t} = \frac{1}{2}\overline \mu_{H,D_t}^2 - \frac{\overline\ch_2^{D_t}}{H^2 r}.$$ When the formula for $\overline \Delta_{H,D_t}$ is fully expanded using the definitions, there are several terms; however, only the term $-\ch_2/(H^2r)$ varies when $c_1$ and $r$ are held fixed and $\ch_2$ is varied.  Thus, the problems of minimizing $\Delta$ and $\overline \Delta_{H,D_t}$ are equivalent.  If the class $c_1({\bf w}_t)/r({\bf w}_t)$ is integral, then $\Delta({\bf w}_t)$ must be zero and ${\bf w}_t$ is a direct sum of $r({\bf w}_t)$ copies of a line bundle.

The only remaining possibility is that $r({\bf w}_t) = 2$ and $n=1$, so that $c_1({\bf w}_t) = (1,-1)$. By Rudakov's classification of the numerical invariants of semistable sheaves \cite{Rudakov}, the smallest discriminant of a semistable sheaf with this rank and first Chern class is $\Delta({\bf w}_t) = \frac{3}{4}$. Thus ${\bf w}_t$ will be the character ${\bf y} = (r,c_1,\Delta) = (2,(1,-1),\frac{3}{4})$. However, we claim that for each $t\in \Q$, $$\overline \Delta_{H,D_t}({\bf y})> \min_{n\in \Z} \overline\Delta_{H,D_t}(\OO_X(n,-n)),$$ so that $\bf y$ is never an extremal character for ${\bf v}$ in the $(H,D_t)$-slice.  To this end, a  computation shows \begin{equation}\label{LineBundleDisc}\tag{$\ast$}\overline \Delta_{H,D_t}(\OO_X(n,-n))= \frac{(n-t)^2}{2}\end{equation} while $$
\overline \Delta_{H,D_t}({\bf y})= \frac{1}{2}+\frac{(t-1)t}{2}.$$
Therefore $\overline \Delta_{H,D_t}({\bf y})\geq \frac{3}{8}$ for all $t$.  On the other hand, if $n$ is an integer closest to $t$, then $\overline \Delta_{H,D_t}(\OO_X(n,-n)) \leq \frac{1}{8}$. 

The remaining statements of the proposition now follow at once from our explicit formula (\ref{LineBundleDisc}) for $\overline \Delta_{H,D_t}(\OO_X(n,-n))$; for fixed $t$ and variable $n$, the discriminant of $\OO_X(n,-n)$ is smallest when $n$ is an integer closest to $t$.   
\end{proof}

Having computed the extremal character ${\bf w}_t$, Corollary \ref{cor-main} determines a boundary nef divisor on $M_H({\bf v})$ corresponding to the $(H,D_t)$-slice, provided that $\overline\Delta_{H,D_t} ({\bf v})$ is sufficiently large (depending on $t$).  Let $W_t$ be the wall in the $(H,D_t)$-slice determined by ${\bf w}_t$, and suppose its center is at the point $(s_{W_t},0)$.  Let ${\bf u}_t$ be the Chern character $${\bf u}_t = ( -1 , s_{W_t}H+D_t,m)\in {\bf v}^\perp,$$ where the number $m$ is determined by the requirement ${\bf u}_t\in {\bf v}^\perp$.   (Note that while ${\bf w}_t$ is technically not well-defined if $t$ is a half-integer, the wall $W_t$ and hence the character ${\bf u}_t$ are independent of the choice.)  Then $\lambda({\bf u}_t)$ spans the ray determined by the wall $W_t$ in the $(H,D_t)$-slice.  

We reiterate that for any given $t\in \Q$, the ray spanned by $\lambda({\bf u}_t)$ will be a boundary nef ray if $\ch_2({\bf v})\ll 0$ depending on $t$.  We compute the center $s_{W_t}$ to more explicitly determine ${\bf u}_t$ in the next result.

\begin{lemma}
If $n$ is a closest integer to $t\in \Q$, then we have $$s_{W_t} = (1-4n)t+2n^2+1+\ch_2({\bf v})$$
\end{lemma}
\begin{proof}
Without loss of generality assume ${\bf w}_t = \ch(\OO_X(n,-n))$.  Then \begin{align*}\overline\mu_{H,D_t}({\bf w}_t) &= 1\\
\overline\mu_{H,D_t}({\bf v}) &= \frac{5}{4}\\
\overline\Delta_{H,D_t}({\bf w}_t) &= \frac{1}{2}(n-t)^2,\\
\overline\Delta_{H,D_t}({\bf v}) &= \frac{1}{2}\left(t-\frac{1}{4}\right)^2 -\frac{\ch_2({\bf v})}{4}
\end{align*} and the result follows at once from the formula for the center.
\end{proof}

As $t$ varies in a unit-length interval $(n-\frac{1}{2},n+\frac{1}{2})$ for some integer $n$, the centers $s_{W_t}$ interpolate linearly between the centers at  half-integer times.  Thus, if $\ch_2({\bf v})$ is sufficiently negative that the divisors $$\lambda({\bf u}_{n-\frac{1}{2}}), \lambda({\bf u}_n), \mbox{ and }\lambda({\bf u}_{n+\frac{1}{2}})$$ are all boundary nef divisors, then $\lambda({\bf u}_t)$ is a boundary nef divisor for \emph{every} $t$ in this interval.  Similarly, if the divisors $$\lambda({\bf u}_{n-\frac{1}{2}}), \lambda({\bf u}_n), \lambda({\bf u}_{n+\frac{1}{2}}), \lambda({\bf u}_{n+1})\mbox{ and } \lambda({\bf u}_{n+\frac{3}{2}})$$ are all boundary nef divisors, then $\lambda({\bf u}_t)$ is a boundary nef divisor for all $t\in (n-\frac{1}{2},n+\frac{3}{2})$.  Furthermore, in this case $\lambda({\bf u}_{n+\frac{1}{2}})$ spans an extremal ray of the nef cone.  By further decreasing $\ch_2({\bf v})$ in this manner, we get the next result.

\begin{proposition}
Let $N>0$ be a positive integer, and suppose $\ch_2({\bf v})$ is sufficiently negative, depending on $N$.  If $t\in [-N,N]$, then $\lambda({\bf u}_t)$ is a boundary nef divisor on $M_H({\bf v})$.  If additionally $t$ is a half-integer, then $\lambda({\bf u}_t)$ spans an extremal ray of the nef cone of $M_H({\bf v})$.  In particular, as $\ch_2({\bf v})$ decreases, the number of extremal rays of $\Nef(M_H({\bf v}))$ becomes arbitrarily large.  
\end{proposition}

In a sense, the ``middle'' portion of the nef cone of $M_H({\bf v})$ stabilizes as $\ch_2({\bf v})$ decreases.

\bibliographystyle{plain}

\begin{thebibliography}{ABCH13}
\bibitem[AB13]{ArcaraBertram}
D. Arcara, A. Bertram. Bridgeland-stable moduli spaces for $K$-trivial surfaces, with an appendix by Max Lieblich, J. Eur. Math. Soc., {\bf 15} (2013), 1--38. 

\bibitem[ABCH13]{ABCH}
D. Arcara, A. Bertram, I. Coskun, and J. Huizenga.
\newblock The minimal model program for the Hilbert scheme of points on $\mathbb{P}^2$ and Bridgeland stability.  Adv. Math., {\bf 235} (2013), 580--626. 




\bibitem[BM14a]{BayerMacri2}
A.~Bayer and E.~Macr\`{i}. Projectivity and birational geometry of Bridgeland
moduli spaces, J. Amer. Math. Soc., {\bf 27} (2014),  707--752.

\bibitem[BM14b]{BayerMacri3}
A. Bayer\ and\ E. Macr\`\i, MMP for moduli of sheaves on K3s via wall-crossing: nef and movable cones, Lagrangian fibrations, Invent. Math. {\bf 198} (2014), no.~3, 505--590.

\bibitem[BC13]{BertramCoskun}
A.~Bertram and I.~Coskun, The birational geometry of the Hilbert scheme of points on surfaces, in {\it Birational geometry, rational curves, and arithmetic},  Simons Symposia, Springer, New York, 2013, 15--55. 

\bibitem[BM15]{BertramMartinez}
A. Bertram and C. Martinez, Change of polarization for moduli of sheaves on surfaces as Bridgeland wall-crossing, preprint.

\bibitem[BMW14]{BertramMartinezWang}
A. Bertram, C. Martinez, and J. Wang, The birational geometry of moduli spaces of sheaves on the projective plane, Geom. Dedicata {\bf 173} (2014), 37--64.


\bibitem[Bo15]{Betal} B. Bolognese, J. Huizenga, Y. Lin, E. Riedl, B. Schmidt, M. Woolf, and X. Zhao, Nef cones of Hilbert schemes of points on surfaces, preprint.  

\bibitem[Br07]{bridgeland:stable}
T. Bridgeland, Stability conditions on triangulated categories, Ann. of Math. (2) {\bf 166} (2007), no.~2, 317--345.

\bibitem[Br08]{Bridgeland}
T. Bridgeland, Stability conditions on $K3$ surfaces, Duke Math. J. {\bf 141} (2008), no.~2, 241--291.



\bibitem[CH14]{CoskunHuizenga}
I.~Coskun and J.~Huizenga, Interpolation, Bridgeland stability and monomial schemes in the plane, J. Math. Pures Appl. {\bf 102} (2014), 930--971.

\bibitem[CH15]{CoskunHuizenga2}
I. Coskun and J. Huizenga, The ample cone of the moduli spaces of sheaves on the plane, to appear in Algebraic Geom.

\bibitem[CHW14]{CoskunHuizengaWoolf}
I.~Coskun, J.~Huizenga and M.~Woolf, The effective cone of the moduli spaces of sheaves on the plane, to appear in J. Eur. Math. Soc. 

\bibitem[GRT15]{Greb}
D.~Greb, J.~Ross, M.~Toma, Semi-continuity of stability for sheaves and variation of Gieseker moduli spaces, preprint.


\bibitem[HL10]{HuybrechtsLehn} D. Huybrechts\ and\ M. Lehn, {\it The geometry of moduli spaces of sheaves}, second edition, Cambridge Mathematical Library, Cambridge Univ. Press, Cambridge, 2010.

\bibitem[La04]{Lazarsfeld} R. Lazarsfeld, {\it Positivity in algebraic geometry I, Classical setting: line bundles and linear series}, Springer-Verlag, 2004. 


\bibitem[LP97]{LePotierLectures}
J. Le Potier, {\it Lectures on vector bundles}, translated by A. Maciocia, Cambridge Studies in Advanced Mathematics, 54, Cambridge Univ. Press, Cambridge, 1997. 

\bibitem[LZ13]{LiZhao}
C. Li and X. Zhao, The MMP for deformations of Hilbert schemes of points on the projective plane, preprint.


\bibitem[Li96]{JunLi}
J. Li, Picard groups of the moduli spaces of vector bundles over algebraic surfaces, in Moduli of vector bundles, Lect. notes in Pure and Appl. Math. {\bf 179} (1996), 129--146.

\bibitem[Ma14]{Maciocia}
A. Maciocia, Computing the walls associated to Bridgeland stability conditions on projective surfaces, Asian J. Math., {\bf 18} no. 2 (2014), 263--279.

\bibitem[MM13]{MM}
A. Maciocia\ and\ C. Meachan, Rank 1 Bridgeland stable moduli spaces on a principally polarized abelian surface, Int. Math. Res. Not. {\bf 2013}, no.~9, 2054--2077.

\bibitem[MW97]{MatsukiWentworth}
K. Matsuki and R. Wentworth, Mumford-Thaddeus principle on the moduli space of vector bundles on an algebraic surface, Internat. J. Math., {\bf 8} no. 1 (1997), 97--148. 


\bibitem[MYY12]{MYY1}
H. Minamide, S. Yanagida, and K. Yoshioka, Fourier-Mukai transforms and the wall-crossing behavior for Bridgeland's stability conditions, preprint.

\bibitem[MYY14]{MYY2}
H. Minamide, S. Yanagida, and K. Yoshioka, Some moduli spaces of Bridgeland's  stability
conditions, Int. Math. Res. Not., to appear.

\bibitem[Mu84]{Mukai}S. Mukai, On the moduli space of bundles on $K3$ surfaces. I, in {\it Vector bundles on algebraic varieties (Bombay, 1984)}, 341--413, Tata Inst. Fund. Res. Stud. Math., 11, Tata Inst. Fund. Res., Bombay.


\bibitem[Nu14]{Nuer}
H. Nuer, Projectivity and birational geometry of moduli spaces of Bridgeland stable objects on an Enriques surface, preprint.


\bibitem[O'G96]{OGrady}
K. O'Grady, Moduli of vector bundles on projective surfaces: Some basic results, Invent. Math. {\bf 123} (1996), 141--207.

\bibitem[Oh10]{Ohkawa}
R. Ohkawa, Moduli of Bridgeland semistable onjects on $\PP^2$. Kodai Math. J. {\bf 33} no. 2 (2010), 329--366.

\bibitem[RS09]{RavindraSrinivas}
G. V. Ravindra, V. Srinivas. The Noether–Lefschetz theorem for the divisor class group. J. Algebra {\bf 322}
(2009), no. 9, 3373-3391.

\bibitem[Ru94]{Rudakov}
A.N. Rudakov, A description of Chern classes of semistable sheaves on a quadric surface, J. Reine Angew. Math. {\bf 453} (1994), 113--135. 

\bibitem[Ry16]{Ryan}
T. Ryan, The effective cone of the moduli space of sheaves on $\P^1\times \P^1$.  Ph.D. Thesis, in preparation.


\bibitem[YY14]{YanagidaYoshioka}
S. Yanagida\ and\ K. Yoshioka, Bridgeland's stabilities on abelian surfaces, Math. Z. {\bf 276} (2014), no.~1-2, 571--610.



\bibitem[Y12]{Yoshioka2}
K. Yoshioka,  Bridgeland's stability and the positive cone of the moduli spaces of stable objects on an abelian surface, preprint. 

\bibitem[Y15]{Yoshioka3}
K. Yoshioka, Wall crossing of the moduli spaces of perverse coherent sheaves on a blowup, preprint.


\end{thebibliography}

\end{document}